\documentclass[12pt,a4paper,psamsfonts]{amsart}

\usepackage{fancyhdr}
\usepackage{appendix}
\usepackage{amssymb,amscd,amsxtra,calc}
\usepackage{mathrsfs}
\usepackage{cmmib57}
\usepackage{multirow}
\usepackage{mathtools}
\usepackage[all]{xy}
\usepackage[colorlinks,linkcolor=blue,anchorcolor=blue,citecolor=green]{hyperref}
\theoremstyle{plain}
    \newtheorem{thm}{Theorem}[section]

    \newtheorem{corollary}[thm]{Corollary}
    \newtheorem{example}[thm]{Example}
    \newtheorem{lemma}[thm]{Lemma}
    \newtheorem{proposition}[thm]{Proposition}
    \newtheorem{question}[thm]{Question}
    \newtheorem{theorem}[thm]{Theorem}

\theoremstyle{definition}
    \newtheorem{definition}[thm]{Definition}

    \newtheorem{remark}[thm]{Remark}
\theoremstyle{remark}

\newcommand{\Q}{\mathbb{Q}}

\newcommand{\R}{\mathbb{R}}

\newcommand{\Z}{\mathbb{Z}}

\newcommand{\alb}{\operatorname{alb}}
\newcommand{\Amp}{\operatorname{Amp}}

\newcommand{\Aut}{\operatorname{Aut}}

\newcommand{\Bigc}{\operatorname{Big}}

\newcommand{\Fix}{\operatorname{Fix}}

\newcommand{\GL}{\operatorname{GL}}

\newcommand{\id}{\operatorname{id}}

\newcommand{\NE}{\overline{\operatorname{NE}}}
\newcommand{\Nef}{\operatorname{Nef}}

\newcommand{\NS}{\operatorname{NS}}

\newcommand{\PE}{\operatorname{PE}}
\newcommand{\Pos}{\operatorname{Pos}}
\newcommand{\Per}{\operatorname{Per}}
\newcommand{\Prep}{\operatorname{Prep}}

\newcommand{\Tor}{\operatorname{Tor}}

\newcommand{\N}{\operatorname{N}}

\newcommand{\Alb}{\operatorname{Alb}}

\newcommand{\Pic}{\operatorname{Pic}}

\begin{document}

\title[Endomorphisms of projective varieties]
{On endomorphisms of projective varieties with numerically trivial canonical divisors}

\author{Sheng Meng}

\address
{
\textsc{Max-Planck-Institut f\"ur Mathematik, Vivatsgasse 7, Bonn 53111, Germany}
}
\email{ms@u.nus.edu}

\begin{abstract}
Let $X$ be a klt projective variety with numerically trivial canonical divisor.
A surjective endomorphism $f:X\to X$ is amplified (resp.~quasi-amplified) if $f^*D-D$ is ample (resp.~big) for some Cartier divisor $D$.
We show that after iteration and equivariant birational contractions, a quasi-amplified endomorphism will descend to an amplified endomorphism.

As an application, when $X$ is Hyperk\"ahler, $f$ is quasi-amplified if and only if it is of positive entropy. In both cases, $f$ has Zariski dense periodic points.
When $X$ is an abelian variety, we give and compare several cohomological and geometric criteria of amplified endomorphisms and endomorphisms with countable and Zariski dense periodic points (after an uncountable field extension).
\end{abstract}
\makeatletter
\@namedef{subjclassname@2020}{\textup{2020} Mathematics Subject Classification}
\makeatother



\maketitle
\tableofcontents

\section{Introduction}

We work over an algebraically closed field $k$ of characteristic $0$.
Let $f$ be a surjective endomorphism of a projective variety $X$. 
It was first defined by Krieger and Reschke (cf.~\cite{KR}) that $f$ is {\it amplified}, if $H:=f^*L-L$ is ample for some Cartier divisor $L$.
The polarized endomorphism is a special case and has been studied in various aspects (cf.~\cite{BH}, \cite{CMZ}, \cite{Meng}, \cite{MZ}, \cite{MZ_PG}, \cite{MZ_S}, \cite{Na-Zh}, \cite{Zh-comp}).
One of the main methods of studying polarized endomorphisms is the equivariant lifting and descending.
However, as discussed in \cite[Section 1]{Meng}, the ``amplified'' property can not be preserved by an equivariant birational lifting and it is not known whether it is preserved via an equivariant descending.
For the lifting problem, one natural way is to imitate the ``quasi-polarized'' endomorphism (cf.~\cite{MZ}) to introduce the {\it quasi-amplified} endomorphism, i.e., $B:=f^*L-L$ is big for some Cartier divisor $L$,
though it is now known that ``quasi-polarized'' is just ``polarized'' (cf.~\cite[Proposition 1.1]{MZ}, \cite[Theorem 5.1]{CMZ}).

From the geometric point of view,
Fakhruddin showed the following very motivating Theorem \ref{thm-fak}.
Further, an amplified endomorphism has only countably many periodic points (cf.~\cite[Lemma 2.4]{Meng}). Note that every variety has only countably many points if the base field is countable. To exclude this, we shall work over an uncountable field.
Now, we define a surjective endomorphism to be {\it PCD}, if its periodic points are countable and Zariski dense after replacing the base field by an uncountable one (cf.~Definition \ref{def-pcd}). 
Alternatively, without the argument of base change, this is equivalent to saying that its periodic points are Zariski dense and the set of fixed points $\Fix(f^i)$ is finite for any $i>0$ (cf.~Proposition \ref{prop-ext}).
Note that ``amplified'' is then always ``PCD'' (cf.~Theorem \ref{thm-amp-pcd}).
On the other hand, when the base field is uncountable, an amplified endomorphism $f$ admits a Zariski dense orbit, which means for some $x\in X$, the orbit $\{f^n(x)\,|\,n\ge 0\}$ is Zariski dense in $X$ (cf.~Theorem \ref{thm-amp-orbit}).
We refer to Section \ref{sec-2} for details.

\begin{theorem}\label{thm-fak}(cf.~\cite[Theorem 5.1]{Fak}) Let $f:X\to X$ be an amplified endomorphism of a projective variety $X$.
Then the set of $f$-periodic points $\Per(f)$ is Zariski dense in $X$.
\end{theorem}

In this paper, we try to look for the connections between cohomological and geometric description among these conditions and mainly focus on a projective variety $X$ of klt Calabi-Yau type, i.e., $(X,\Delta)$ is klt and $K_X+\Delta\equiv 0$ (numerical equivalence) for some effective Weil $\Q$-divisor $\Delta$.
Note that such a pair has $K_X+\Delta\sim_{\mathbb{Q}} 0$ ($\mathbb{Q}$-linear equivalence) by \cite[Chapter V, Corollary 4.9]{ZDA}.
We refer to \cite{KM} for the standard definitions, notation, and terminology in birational geometry.
In this setting, we first show that a quasi-amplified endomorphism $f$ of such an $X$ is {\it birationally equivalent} to an amplified endomorphism $g:Y\to Y$, which means there is a birational map $\pi:X\dasharrow Y$ such that $g\circ\pi=\pi\circ f$.
Precisely, we have the following result.
 
\begin{theorem}\label{main-thm-qa-a}(cf.~Theorem \ref{thm-qa-a})
Let $f:X\to X$ be a quasi-amplified endomorphism of a projective variety $X$ of klt Calabi-Yau type.
Then replacing $f$ by a positive power, there is an $f$-equivariant sequence of birational contractions of extremal rays (in the sense of \cite[Definition 4.1]{MZ_PG})
$$X=X_1\to \cdots \to X_i \to \cdots \to X_r$$
(i.e. $f=f_1$ descends to surjective endomorphism $f_i$ on each $X_i$), such that we have:
\begin{itemize}
\item[(1)] $f_r$ is amplified.
\item[(2)] For each $i$, $X_i$ is of klt Calabi-Yau type.
\item[(3)] For each $i$, $f_i$ is of positive entropy (cf.~Definition \ref{def-pcd}) and quasi-amplified and $\Per(f_i)\cap U_i$ is countable and Zariski dense in $X_i$ for some open subset $U_i\subseteq X_i$. 
\item[(4)] Suppose the base field $k$ is uncountable. For each $i$, $f_i$ has a Zariski dense orbit.
\end{itemize}
\end{theorem}

Let $f:X\to X$ be an amplified endomorphism of a projective variety $X$.
Then $X$ has Kodaira dimension $\kappa(X)\le 0$ by taking the equivariant Iitaka fibration (cf.~\cite[Lemma 2.5]{Meng} and \cite[Theorem A]{NZ09}).
Suppose further that $X$ is a smooth complex projective variety with $K_X$ numerically trivial (hence $\kappa(X)=0$).
We may apply Beauville-Bogomolov decomposition (cf.~\cite{Bea}) and have an \'etale cover $(\prod_i X_i)\times A\to X$ where $A$ is an abelian variety and $X_i$ is either a projective Hyperk\"ahler manifold or a strict Calabi-Yau variety.
Moreover, the cover can be chosen as the so called Albanese closure (cf.~\cite[Lemma 2.12]{Na-Zh}) such that $f$ can be lifted equivariantly and then the iteration $f^t$ splits as a product of surjective endomorphisms on $A$ and $X_i$ for some $t>0$ (cf.~\cite[Theorem 4.6]{San}).
We also refer to \cite[Proposition 3.5]{Na-Zh} for a version of the singular case (cf.~Proposition \ref{prop-nz}).

The following result is an application of Theorem \ref{main-thm-qa-a} to the Hyperk\"ahler case. Note that in this case, any surjective endomorphism is an automorphism. 
\begin{theorem}\label{main-thm-hk}
Let $f:X\to X$ be an automorphism of a projective Hyperk\"ahler manifold $X$.
Then the following are equivalent.
\begin{itemize}
\item[(1)] $f$ is of positive entropy.
\item[(2)] $f^*D\not\equiv D$ for any nef $\mathbb{R}$-Cartier divisor $D\not\equiv 0$.
\item[(3)] $f$ is quasi-amplified.
\item[(4)] For some $n>0$, $f^n$ is birationally equivalent to some amplified automorphism $f':X'\to X'$.
\end{itemize}
Moreover, if $f$ is PCD, then all the above are satisfied.
\end{theorem}

\begin{remark}
\begin{enumerate}
\item Oguiso \cite[Theorem 4.1]{Ogu15} constructed an automorphism $f:S\to S$ of a projective K3 surface $S$ with Picard number $2$, such that no eigenvalue of $f^*|_{\NS(S)}$ is $1$.
In particular, $f$ is amplified by noting that $f^*-\id$ on $\NS(S)\otimes_{\mathbb{Z}}\mathbb{R}$ is surjective. 
We refer to Corollary \ref{cor-k3} and Examples \ref{exa-k3-1} and \ref{exa-k3-2} for further discussion on the case of projective K3 surfaces.
\item There are three types of automorphisms, characterized by the behaviour of the linear automorphism $\varphi:=f^*|_{H^*(X,\mathbb{Z})}$ (cf.~\cite{Can14}).
If $\varphi$ has finite order, then $f$ is elliptic.
Otherwise, $f$ is either parabolic or loxodromic: it is parabolic if $\varphi$ has infinite order, but none of its eigenvalues has modulus $>1$; it is loxodromic if some eigenvalue of $\varphi$ has modulus $>1$.
Note that loxodromic automorphism is equivalent to automorphism of positive entropy.
\end{enumerate}
\end{remark}

Next, we consider another important case: the abelian varieties.
In this case, ``quasi-amplified'' is just ``amplified'' since any big divisor of an abelian variety is ample.
Krieger and Reschke \cite[Proposition 2.5]{KR} gave a characterization of PCD isogenies. 
Here, we provide a similar criterion of amplified endomorphisms for comparison. 

\begin{theorem}\label{main-thm-ab1}(cf.~Theorems \ref{thm-ab-amp-cri} and \ref{thm-pcd-cri})
Let $f:A\to A$ be a surjective endomorphism of an abelian variety $A$.
Then the following holds.
\begin{itemize}
\item[(1)] $f$ is amplified if and only if no eigenvalue of $f^*|_{H^1(A,\mathcal{O}_A)}$ is of modulus $1$.
\item[(2)] $f$ is PCD if and only if no eigenvalue of $f^*|_{H^1(A,\mathcal{O}_A)}$ is a root of unity.
\end{itemize}
\end{theorem}

\begin{remark}
When $A$ is an abelian surface, Krieger and Reschke \cite[Propositions 2.5 and 4.3]{KR} showed that an isogeny $f$ is amplified if and only if the set of preperiodic points $\Prep(f)=\Tor(A)$ where $\Tor(A)$ is the set of torsion points in $A$, which is equivalent to saying that $f$ is PCD.
The key reason is due to the phenomenon that any eigenvalue of $f^*|_{H^1(A,\mathcal{O}_A)}$ of modulus $1$ is necessarily a root of unity for when $\dim(A)=2$.
Indeed, let $\alpha, \beta$ be the eigenvalues of $f^*|_{H^1(A,\mathcal{O}_A)}$.
Then the eigenvalues of $f^*|_{H^1(A,\mathbb{Z})}$ are $\alpha, \beta, \overline{\alpha}, \overline{\beta}$.
Once some of them has modulus $1$, so are all of them and hence all are roots of unity by Kronecker's theorem.
However, in the higher dimensional cases, this phenomenon no longer holds due to the existence of the Salem polynomials.
So we can construct a PCD endomorphism which is not amplified induced by any Salem polynomial; see Example \ref{exa-salem}.
\end{remark}

The following result gives another characterizations and comparison of  amplified and PCD endomorphisms in terms of divisors.
\begin{theorem}\label{main-thm-ab2}(cf.~Theorems \ref{thm-ab-amp-cri} and \ref{thm-pcd-cri2})
Let $f:A\to A$ be a surjective endomorphism of an abelian variety $A$.
Then the following hold.
\begin{itemize}
\item[(1)] $f$ is amplified if and only if $f^*D\not\equiv D$ for any nef $\R$-Cartier divisor $D\not\equiv 0$.
\item[(2)] $f$ is PCD if and only if $f^*D\not\equiv D$ for any nef Cartier divisor $D\not\equiv 0$.
\end{itemize}
\end{theorem}

Let $f:X\to X$ be a surjective endomorphism of a projective variety $X$ over $k$.
When the base field $k$ is uncountable, Amerik and Campana \cite{AC} showed that $f$ has a Zariski dense orbit if and only if there is no dominant rational map $\pi:X\dasharrow \mathbb{P}^1$ such that $\pi\circ f=f$.
When $k$ is countable, this equivalence remains unknown (cf.~\cite[Conjecture 7.14]{MS}) except the case when $X$ is an abelian variety proved by Ghioca and Scanlon \cite[Theorem 1.2]{GS}.
In the following, we show that they are also equivalent to ``PCD endomorphisms'' for self-isogenies of abelian varieties.

\begin{theorem}\label{main-thm-ab3}
Let $f:A\to A$ be an isogeny of an abelian variety $A$.
Then the following are equivalent.
\begin{itemize}
\item[(1)] $f$ is PCD.
\item[(2)] $f$ has a Zariski dense orbit.
\item[(3)] There is no dominant rational map $\pi:A\dasharrow \mathbb{P}^1$ such that $\pi\circ f=f$.
\end{itemize}
The equivalence of (2) and (3) are due to Ghioca and Scanlon \cite[Theorem 1.2]{GS}.
\end{theorem}

Finally, we can show that the PCD and amplified properties can be preserved via the Albanese map in the following setting, which gives a partial answer to \cite[Question 1.10]{KR} (cf.~Question \ref{que-des}).
We refer to \cite[Theorem 1.2 and Section 5]{CMZ} for the case of polarized endomorphisms.

\begin{theorem}\label{main-thm-des}
Let $f:X\to X$ be a PCD (resp.~quasi-amplified) surjective endomorphism of a klt projective variety $X$ with $K_X\equiv 0$.
Then the Albanese morphism $\alb_X:X\to \Alb(X)$ is surjective with $(\alb_X)_\ast \mathcal{O}_X=\mathcal{O}_{\Alb(X)}$. 
Furthermore, the induced endomorphism $g:=f|_{\Alb(X)}:\Alb(X)\to \Alb(X)$ is PCD (resp.~amplified).
\end{theorem}


\section{Preliminaries}\label{sec-2}

Let $X$ be a projective variety of dimension $n$.
Throughout this paper, by a Cartier divisor, we always mean an integral Cartier divisor.
We refer to \cite[Definitions 2.1 and 2.2]{MZ} for the numerical equivalence ($\equiv$) of $\R$-Cartier divisors and the weak numerical equivalence ($\equiv_w$) of $r$-cycles.
Denote by $\N^1(X):=\NS(X) \otimes_{\Z} \mathbb{R}$ for the N\'eron-Severi group $\NS(X)$.
One can also regard $\N^1(X)$ as the quotient vector space of $\R$-Cartier divisors modulo the numerical equivalence.
Denote by $\N_r(X)$ the quotient vector space of $r$-cycles modulo the weak numerical equivalence.

\begin{definition}\label{def-cones} Let $X$ be a projective variety. We define:
\begin{itemize}
\item $\Amp(X)$, the cone of classes of ample $\mathbb{R}$-Cartier divisors in $\N^1(X)$.
\item $\Nef(X)$, the cone of classes of nef $\mathbb{R}$-Cartier divisors in $\N^1(X)$.
\item $\Bigc(X)$, the cone of classes of big $\mathbb{R}$-Cartier divisors in $\N^1(X)$.
\item $\PE^1(X)$, the closure of the cone of classes of effective $\mathbb{R}$-Cartier divisors in $\N^1(X)$.
\item $\NE(X)$, the closure of the cone of classes of effective $1$-cycles with $\R$-coefficients in $\N_1(X)$.
\end{itemize}
\end{definition}
Let $f:X\to X$ be a surjective endomorphism of a projective variety $X$ over an algebraically closed field $k$.
Thanks to the finiteness of $f$, all the above cones are $f^*$ and $f_*$ invariant by the projection formula.
Moreover, $f^*f_*=f_*f^*=(\deg f) \id$ on $\N^1(X)$ and $\N_r(X)$ for any $r$;
see \cite[Section 2.3]{Zh-comp}.

Denote by $$\Fix(f):=\{x\in X\,|\, f(x)=x\}$$ the set of fixed points of $f$.
Denote by $$\Per(f):=\bigcup_{i=1}^{+\infty} \Fix(f^i)$$ the set of periodic points of $f$.
Denote by $$\Prep(f):=\bigcup_{i=1}^{+\infty} f^{-i}(\Per(f))$$ the set of preperiodic points of $f$.

Let $K/k$ be a field extension such that $K$ is algebraically closed.
Denote by $X_K:=X\times_k K$ and $f_K:X_K\to X_K$ the induced surjective endomorphism.
The following definitions (4) and (5) coincide with the usual one (cf.~\cite{DS04} and \cite[\S 4]{DS17}) when $X$ is smooth and defined over $\mathbb{C}$.

\begin{definition}\label{def-pcd}
Let $f:X\to X$ be a surjective endomorphism of a projective variety $X$.
\begin{itemize}
\item[(1)] $f$ is {\it amplified} if $f^*D-D$ is an ample Cartier divisor for some Cartier divisor $D$.
\item[(2)] $f$ is {\it quasi-amplified} if $f^*D-D$ is a big Cartier divisor for some Cartier divisor $D$.
\item[(3)] $f$ is {\it PCD} if $\Per(f_K)$ is countable and Zariski dense in $X_K$ for some uncountable algebraically closed field extension $K/k$.
\item[(4)] $f$ is of {\it positive entropy} if the spectral radius of $f^*|_{\N^1(X)}$ is greater than $1$.
\item[(5)] $f$ is of {\it null entropy} if $f$ is not of positive entropy.
\end{itemize}
\end{definition}

The following result is frequently used throughout this paper.
\begin{proposition}\label{prop-ext}
Let $f:X\to X$ be a surjective endomorphism of a projective variety $X$ over $k$.
Then the following hold.
\begin{enumerate}
\item  $f$ is amplified (resp.~quasi-amplified) if and only if $f^*D-D$ is an ample (resp.~big) $\mathbb{R}$-Cartier divisor for some $\mathbb{R}$-Cartier divisor $D$.
\item  $f$ is amplified (resp.~quasi-amplified, of positive entropy) if and only if so is $f_K$ for any algebraically closed field extension $K/k$.
\item  $f$ is PCD if and only if $\Per(f)$ is Zariski dense in $X$ and $\Fix(f^i)$ is finite for any $i>0$. 
\item  $f$ is PCD if and only if $f_K$ is PCD for any algebraically closed field extension $K/k$.
\item  For any positive integer $n>0$, $f$ is amplified (resp.~quasi-amplified, PCD) if and only if so is $f^n$.
\end{enumerate}
\end{proposition}
\begin{proof}
(1) One direction is clear. Suppose $f^*D-D$ is an ample $\mathbb{R}$-Cartier divisor for some $\mathbb{R}$-Cartier divisor $D$.
Note that $\Amp(X)$ is an open cone in $\N^1(X)$ and $(f^*-\id)|_{\N^1(X)}$ is continuous.
Then there exists some $\Q$-Cartier divisor $D'$ with $[D']$ in a sufficiently small neighbourhood of $[D]$ such that $[f^*D'-D']\in \Amp(X)$.
Assume that $mD'$ is Cartier for some $m>0$.
Then $f^*mD'-mD'$ is an ample Cartier divisor.
The quasi-amplified case is similar.

(2) Let $\pi: X_K\to X$ be the projection.
Note that $\Pic(X_K)=\Pic(X)\times_k K$ and $\Pic^0(X_K)=\Pic^0(X)\times_k K$.
Then $\pi^*:\NS(X)\to \NS(X_K)$ is an isomorphism and $\pi^*\circ f^*=f_K^*\circ\pi^*$.
Moreover, $\pi^*(\Amp(X))=\Amp(X_K)$ and $\pi^*(\Bigc(X))=\Bigc(X_K)$.
Then (2) is clear.

(3) Let $K/k$ be an algebraically closed field extension.
Let $\Delta_X$ be the diagonal of $X\times X$ and $\Gamma_f$ the graph of $f$.
We can identify $\Fix(f_K)$ with $\Delta_{X_K}\cap \Gamma_{f_K}$.
In particular, $\Fix(f_K)$ is defined over $k$ and hence $\Fix(f_K)=\Fix(f)\times_k K$.
Then $\Fix(f^i)$ is finite if and only if so is  $\Fix(f_K^i)$ for any $i>0$.
Note that $\overline{\Per(f_K)}=\overline{\bigcup_{i=1}^{+\infty} \Fix(f^i)\times_k K}\subseteq \overline{\Per(f)}\times_k K\subseteq \overline{\Per(f_K)}$ where the last inclusion is by Lemma \ref{lem-ext}.
So we have $\overline{\Per(f_K)}=\overline{\Per(f)}\times_k K$.
Now one direction is clear.

Suppose $\Per(f_K)$ is countable and Zariski dense in $X_K$ for some uncountable field extension $K/k$.
Then $\Per(f)$ is Zariski dense in $X$.
We claim that $\Fix(f_K^i)$ is finite for any $i>0$.
Otherwise, $\Fix(f_K^i)$ is infinite for some $i>0$.
Let $Z$ be the closure of  $\Fix(f_K^i)$ in $X_K$.
Then $f_K^i|_Z=\id_Z$ and hence $Z\subseteq \Fix(f_K^i)\subseteq \Per(f_K)$.
However, $K$ being uncountable and $\dim(Z)>0$ imply that $Z$ is uncountable, a contradiction.
Therefore, $\Fix(f^i)$ is finite for each $i>0$.
So (3) is proved.

(4) A similar argument to (3) works.

(5) Let $\varphi:=f^*|_{\N^1(X)}$. Note that $\varphi^n-\id=(\varphi-\id)\circ \sum_{i=0}^{n-1}\varphi^i=\sum_{i=0}^{n-1}\varphi^i\circ (\varphi-\id)$.
Suppose $H:=\varphi(D)-D$ is ample (resp.~big). 
Then $\varphi^n(D)-D=\sum_{i=0}^{n-1}\varphi^i(H)$ is ample (resp.~big).
Conversely, suppose $\varphi^n(D)-D$ is ample (resp.~big).
Then $\varphi(D')-D'=\varphi^n(D)-D$ is ample (resp.~big) where $D':=\sum_{i=0}^{n-1}\varphi^i(D)$.
Finally, note that $\Per(f^n)=\Per(f)$ always holds true.
So (5) is proved.
\end{proof}

\begin{lemma}\label{lem-ext}
Let $K/k$ be algebraically closed fields.
Let $S$ be a subset of $\mathbb{P}_k^n$ and regard $\mathbb{P}_k^n$ as a subset of $\mathbb{P}_K^n$.
Denote by $\overline{S}^k$ the closure of $S$ in $\mathbb{P}_k^n$ and $\overline{S}^K$ the closure of $S$ in $\mathbb{P}_K^n$.
Let $f\in K[x_0,\cdots, x_n]$ be a homogeneous polynomial such that $f|_S=0$.
Then $f|_{\overline{S}^k}=0$.
In particular, $\overline{S}^k\subseteq \overline{S}^K$.
\end{lemma}
\begin{proof}
We may write $f:=\sum_{i=1}^m a_if_i$ such that $f_i$ are homogeneous polynomials of the same degree with coefficients in $k$ and $a_i$ are $k$-linearly independent.
For any $s\in S$, $f(s)=0$ and hence $f_i(s)=0$ for all $i$.
Then $f_i|_{\overline{S}^k}=0$ for all $i$.
Therefore, $f|_{\overline{S}^k}=0$.
\end{proof}

Now we may rewrite Fakhruddin's Theorem \ref{thm-fak} in the following way.
Indeed, by Proposition \ref{prop-ext}, we may work over an uncountable field, then uncountably many periodic points will produce a positive dimensional pointwise fixed subvariety (after iteration), however the restriction of the amplified endomorphism on this subvariety is then both identity and amplified, a contradiction.
Such statement is already in the proof of Fakhruddin's theorem, though not explicitly stated.
\begin{theorem}[Fakhruddin]\label{thm-amp-pcd}
Let $f:X\to X$ be an amplified endomorphism of a projective variety $X$.
Then $f$ is PCD.
\end{theorem}

Let $f:X\dashrightarrow X$ be a dominant rational self-map of a projective variety $X$.
We say $f$ has a {\it Zariski dense orbit}, if for some $x\in X$, the orbit $\{f^n(x)\,|\,n\ge 0\}$ is Zariski dense in $X$.
We recall the following useful result proved by Amerik and Campana. Here, we rewrite it a bit and only consider surjective endomorphisms for convenience.
Note that it still remains unknown without the ``uncountable'' assumption; see \cite[Conjecture 7.14]{MS}.
\begin{theorem}(cf.~\cite{AC})\label{thm-AC}
Let $f:X\dashrightarrow X$ be a dominant rational self-map of a projective variety $X$ over an uncountable algebraically closed field $k$.
Then $f$ has no Zariski dense orbit if and only if there is a dominant rational map $\pi:X\dasharrow \mathbb{P}^1$ such that $\pi\circ f=f$.
\end{theorem}

The following is an application of the above result, which is originally motivated by Zhang \cite[Conjecture 4.1.6]{Zhsw} for polarized endomorphisms.
\begin{theorem}\label{thm-amp-orbit}
Let $f:X\to X$ be an amplified endomorphism of a projective variety $X$ over an uncountable algebraically closed field $k$.
Then $f$ has a Zariski dense orbit.
\end{theorem}
\begin{proof}
Suppose $f$ has no Zariski dense orbit.
By Theorem \ref{thm-AC}, there is a dominant rational map $\pi:X\dasharrow \mathbb{P}^1$ such that $\pi\circ f=\pi$.
Let $U$ be an open dense subset of $X$ such that $\pi$ is well defined over $U$.
Denote by $X_y:=\overline{\pi|_U^{-1}(y)}$ for any $y\in \pi(U)$.
Then $f^{-1}(X_y)=X_y$.
Note that $f|_{X_y}$ is amplified and $\dim(X_y)>0$.
By Theorem \ref{thm-amp-pcd}, $\Per(f|_{X_y})\cap U\neq \emptyset$.
Note that $\Per(f)\supseteq \bigcup_{y\in \pi(U)} (\Per(f|_{X_y})\cap U)$ and the latter one is an uncountable disjoint union.
So $\Per(f)$ is uncountable, a contradiction.
\end{proof}

One can see easily that if Theorem \ref{thm-AC} holds true without the ``uncountable'' assumption, then so does Theorem \ref{thm-amp-orbit}.
Indeed, a positive answer to the following question is enough to show that  Theorem \ref{thm-amp-orbit} (in particular \cite[Conjecture 4.1.6]{Zhsw})  holds true without the ``uncountable'' assumption.
\begin{question}
Let $f:X\to X$ be a surjective endomorphism of a projective variety $X$ over a countable algebraically closed field $k$.
Suppose $f_K$ has a Zariski dense orbit for some algebraically closed field extension $K/k$.
Will $f$ also admit a Zariski dense orbit?
\end{question}

\section{General results on surjective endomorphisms}

\begin{proposition}\label{prop-a-cri}
Let $f:X\to X$ be a surjective endomorphism of a projective variety $X$.
Then the following are equivalent.
\begin{itemize}
\item[(1)] $f$ is amplified.
\item[(2)] For any $Z\in \NE(X)$, $f_*Z\equiv_w Z$ implies $Z\equiv_w 0$.
\end{itemize}
\end{proposition}
\begin{proof}
Suppose $f^*D-D$ is ample.
For any pseudo-effective $1$-cycle $Z\not\equiv_w 0$, we have $(f^*D-D)\cdot Z=D\cdot (f_*Z-Z)>0$ and hence $f_*Z\not\equiv_w Z$.
So (1) implies (2).
Suppose $f$ is not amplified.
Let $V$ be the image of $f^*|_{\N^1(X)}-\id$.
Then $V\cap \Amp(X)=\emptyset$ and hence there exists some $1$-cycle $Z\not\equiv_w 0$ such that $L\cdot Z=0$ for any $L\in V$ and $A\cdot Z>0$ for any $A\in \Amp(X)$.
If $Z\not\in \NE(X)$, then there exists some Cartier divisor $B$ such that $B\cdot Z\le 0$ and $B\cdot C>0$ for any $C\in \NE(X)$.
By Kleiman's ampleness criterion (cf.~\cite[Theorem 1.18]{KM}, $B\in \Amp(X)$.
So we get a contradiction and hence $Z\in \NE(X)$.
Note that $D\cdot (f_*Z-Z)=(f^*D-D)\cdot Z=0$ for any Cartier divisor $D$.
Therefore, $f_*Z\equiv_w Z$.
\end{proof}


\begin{lemma}\label{lem-np-aut}
Let $f:X\to X$ be a surjective endomorphism of null entropy of a projective variety $X$.
Then all the eigenvalues of $f^*|_{\N^1(X)}$ are roots of unity and $f$ is an automorphism.
\end{lemma}
\begin{proof}
Note that all the eigenvalues of $f^*|_{\NS(X)}$ are algebraic integers of modulus $\le 1$ and $0\neq \det f^*|_{\NS(X)}\in \mathbb{Z}$.
If some eigenvalue has modulus $<1$, then $0<|\det f^*|_{\NS(X)}|<1$, a contradiction.
So all eigenvalues are of modulus $1$.
By Kronecker's theorem,
we may further assume all the eigenvalues of $f^*|_{\N^1(X)}$ are $1$, after replacing $f$ by a positive power.
Let $x_1,\cdots, x_r$ be a basis of $\N^1(X)$ such that either $f^*x_i=x_i$ or $f^*x_i=x_i+x_{i+1}$.
Let $(a_1,\cdots, a_r)$ be a sequence of non-negative integers such that $\sum\limits_{i=1}^ra_i=\dim(X)$.
We define a partial order that $(a_1,\cdots, a_r)<(b_1,\cdots, b_r)$, if for some $k$, $a_k<b_k$ and $a_i\le b_i$ for any $i\ge k$.
Let $(a_1,\cdots, a_r)$ be the maximal one such that $x_1^{a_1}\cdots x_r^{a_r}\neq 0$.
Then $(\deg f)x_1^{a_1}\cdots x_r^{a_r}=(f^*x_1)^{a_1}\cdots (f^*x_r)^{a_r}=x_1^{a_1}\cdots x_r^{a_r}\neq 0$.
So $\deg f=1$.
\end{proof}

\begin{lemma}\label{lem-amp-pos}
Let $f:X\to X$ be an amplified endomorphism of a projective variety $X$.
Then $f$ is of positive entropy.
\end{lemma}
\begin{proof}
Suppose $f$ is of null entropy.
By Lemma \ref{lem-np-aut}, $f$ is an automorphism.
By the projection formula, $f_*|_{\N_1(X)}$ is the dual action of $f^*|_{\N^1(X)}$.
In particular, all the eigenvalues of $f_*|_{\N_1(X)}$ are of modulus $1$. 
By the Perron-Frobenius theorem, $f_*Z\equiv_w Z$ for some $Z\in \NE(X)\backslash\{0\}$, a contradiction by Proposition \ref{prop-a-cri}.
\end{proof}

\begin{lemma}\label{lem-pcd-pos}
Let $f:X\to X$ be a PCD endomorphism of a smooth projective variety $X$.
Then $f$ is of positive entropy.
\end{lemma}
\begin{proof}
Let $K$ be a finitely generated field over $\mathbb{Q}$ such that $f:X\to X$ is defined over $K$.
Then there is a field extension $\mathbb{C}/K$. 
So we may assume $X$ is defined over $\mathbb{C}$ by Proposition \ref{prop-ext}.
Suppose $f$ is of null entropy.
By Lemma \ref{lem-np-aut}, $f$ is an automorphism.
By \cite[Propositions 3.5 and 3.6]{Dinh12} and Kronecker's theorem, we may assume all the eigenvalues of $f^*|_{H^i(X,\mathbb{Z})}$ are $1$ for each $i$ after replacing $f$ by a positive power.
Since $f$ is PCD, $\Fix(f^n)$ is finite for any $n>0$ by Proposition \ref{prop-ext}.
Applying the Lefschetz fixed point formula, we have
$$\sharp \Fix(f^n)\le \sum\limits_i (-1)^i tr((f^n)^*|_{H^i(X,\mathbb{C})})=\sum\limits_i (-1)^ih^i(X,\mathbb{C})=e(X),$$
where $tr$ is the trace, $e(X)$ is the Euler characteristic of $X$, and $\sharp \Fix(f^n)$ counts $\Fix(f^n)$ without multiplicities.
However, $\Per(f)$ is infinite and hence the set $\{\sharp \Fix(f^n)\,|\, n>0\}$ has no upper bound, a contradiction. 
\end{proof}

In general, we ask the following question. 
\begin{question}\label{que-pos}
Let $f:X\to X$ be a quasi-amplified or PCD endomorphism of a projective variety $X$.
Is $f$ of positive entropy?
\end{question}

In the rest of this section, we consider the lifting and descending problems.
\begin{lemma}\label{lem-fin-lift-des}
Let $\pi:X\to Y$ be a finite surjective morphism of projective varieties.
Let $f:X\to X$ and $g:Y\to Y$ be surjective endomorphisms such that $g\circ \pi=\pi\circ f$.
Then $f$ is PCD if and only if so is $g$.
\end{lemma}
\begin{proof}
By Proposition \ref{prop-ext}, we may work assume the base field is uncountable.
Note that $\pi(\Per(f))\subseteq \Per(g)$.
For any $y\in \Per(g)$, since $\pi^{-1}(y)$ is finite, there exists some $x\in \Per(f)\cap \pi^{-1}(y)$.
So $\pi(\Per(f))=\Per(g)$.
Clearly, $\Per(f)$ is countable and Zariski dense if and only if so is $\Per(g)$.
\end{proof}

\begin{lemma}\label{lem-prod}
Let $f:X\to X$ and $g:Y\to Y$ be surjective endomorphisms of projective varieties.
Then $f\times g$ is PCD (resp.~amplified, quasi-amplified) if and only if so are $f$ and $g$.
\end{lemma}
\begin{proof}
Note that $\Per(f\times g)=\Per(f)\times \Per(g)$.
Clearly, $\Per(f\times g)$ is countable and Zariski dense if and only if so are $\Per(f)$ and $\Per(g)$.

Suppose $f^*D_X-D_X=A_X$ and $g^*D_Y-D_Y=A_Y$ with $A_X$ and $A_Y$ being ample (resp.~big).
Let $D=p_X^*D_X+p_Y^*D_Y$ where $p_X$ and $p_Y$ are the natural projections.
Then $(f\times g)^*D-D=p_X^*f^*D_X+p_Y^*g^*D_Y-p_X^*D_X-p_Y^*D_Y=p_X^*A_X+p_Y^*A_Y$ is ample (resp.~big).

Suppose $B=(f\times g)^*D-D$ is big.
Write $B=A+E$ where $A$ is ample and $E$ is effective.
For general $y\in Y$, $X\times \{y\}$ is not contained in the support of $E$.
Then $B|_{X\times \{y\}}$ is big and hence $f$ is quasi-amplified.
Similarly, so is $g$.
In particular, when $E=0$, both $f$ and $g$ are then amplified.
\end{proof}

\begin{lemma}\label{lem-cone-lim-pos} Let $f:V\to V$ be an invertible linear map of a positive dimensional real normed vector space $V$ such that $f(C)=C$ for a closed convex cone $C\subseteq V$ which spans $V$ and contains no line.
Suppose $x\in C^\circ$  (the interior part of $C$) and $y:=\lim\limits_{n\to +\infty} \frac{f^n(x)}{|f^n(x)|}$ exists.
Then $f(y)=ry$ where $r$ is the spectral radius of $f$.
\end{lemma}
\begin{proof}
By the Perron-Frobenius theorem, $f(x_1)=rx_1$ for some $x_1\in C$.
By the assumption, $f(y)=\lim\limits_{n\to +\infty} \frac{f^{n+1}(x)}{|f^n(x)|}=(\lim\limits_{n\to +\infty} \frac{|f^{n+1}(x)|}{|f^n(x)|})y$.
Then $a:=\lim\limits_{n\to +\infty} \frac{|f^{n+1}(x)|}{|f^n(x)|}$ exists and $a\le r$.
Suppose $a<r$.
Then $\lim\limits_{n\to +\infty} \frac{|f^n(x)|}{r^n}=0$.
Since $x\in C^\circ$, $x-\epsilon x_1\in C$ for some $\epsilon>0$.
We have $\lim\limits_{n\to +\infty} \frac{|f^n(x)|}{|f^n(x-\epsilon x_1)|}=\lim\limits_{n\to +\infty}\frac{1}{|\frac{f^n(x)}{|f^n(x)|}-\frac{\epsilon r^n x_1}{|f^n(x)|}|}=\lim\limits_{n\to +\infty}\frac{1}{|y-\frac{\epsilon r^n x_1}{|f^n(x)|}|}=0.$
Then $\lim\limits_{n\to +\infty} \frac{f^n(x-\epsilon x_1)}{|f^n(x-\epsilon x_1)|}
=\lim\limits_{n\to +\infty} \frac{-\epsilon r^n x_1}{|f^n(x-\epsilon x_1)|}
=\lim\limits_{n\to +\infty} \frac{-\epsilon x_1}{|0-\epsilon x_1)|}
=-\frac{x_1}{|x_1|}$.
Note that $C$ contains no line and $f(C)=C$.
We get a contradiction.
\end{proof}

\begin{lemma}\label{lem-lift-des}
Let $\pi:X\to Y$ be a generically finite surjective morphism of projective varieties.
Let $f:X\to X$ and $g:Y\to Y$ be surjective endomorphisms such that $g\circ \pi=\pi\circ f$.
Then the following are true.
\begin{itemize}
\item[(1)] If $g$ is quasi-amplified, then so is $f$.
\item[(2)] $f$ is of positive entropy if and only if so is $g$.
\end{itemize}
\end{lemma}
\begin{proof}
Suppose $g^*E-E$ is big.
Let $F:=\pi^*E$.
Then $f^*F-F=\pi^*(g^*E-E)$ is big since $\pi$ is generically finite.
So (1) is true.

For (2), one direction is trivial.
Suppose $f$ is of positive entropy and $g$ is of null entropy.
Replacing $g$ by a positive power, we may assume all the eigenvalues of $g^*|_{\N^1(Y)}$ are $1$.
Fix a norm on $\N^1(Y)$.
Let $B$ be a big Cartier divisor of $Y$.
Then $D:=\lim\limits_{n\to +\infty} \frac{(g^n)^*B}{|(g^n)^*B|}\in \PE^1(Y)$ exists and $g^*D\equiv D$.
Let $B':=\pi^*B$ which is big since $\pi$ is generically finite. 
Then $D':=\lim\limits_{n\to +\infty} \frac{(f^n)^*B'}{|(f^n)^*B'|}=\pi^*D\in \PE^1(X)$ exists and $f^*D'\equiv D'$, a contradiction by Lemma \ref{lem-cone-lim-pos}.
\end{proof}

In general, we ask the following question. We shall see later that it is true for the case of abelian varieties (cf.~Proposition \ref{prop-5lem}).
\begin{question}\label{que-des}
Let $\pi:X\to Y$ be a surjective morphism of projective varieties.
Let $f:X\to X$ and $g:Y\to Y$ be surjective endomorphisms such that $g\circ \pi=\pi\circ f$.
Suppose $f$ is amplified (resp.~quasi-amplified).
Will $g$ be amplified (resp.~quasi-amplified)?
\end{question}

\section{Equivariant contractions for quasi-amplified endomorphisms}

Let $V$ be a positive dimensional real vector space. 
Let $\mathcal{C}$ be a closed convex cone which spans $V$ and contains no line.
Let $F$ be a closed convex subcone of $\mathcal{C}$.
We say $F$ is an {\it extremal face} of $\mathcal{C}$ if for any $x, y\in \mathcal{C}$, $x+y\in F$ implies $x, y\in F$.
For any $0\neq x\in V$, denote by $$R_x:=\{ax\,|\,a\ge 0\}$$ the ray generated by $x$.
Let $R$ be a ray of $\mathcal{C}$.
We say an extremal ray $R$ is {\it isolated} if there exists a nonzero $x\in R$ and an open neighborhood $x\in U$, such that for any $y\in U$, either $y\in R$ or the ray $R_y$ generated by $y$ is not extremal.

\begin{lemma}\label{lem-fix-ray} Let $f:V\to V$ be an invertible linear map of a positive dimensional real normed vector space $V$ such that $f(\mathcal{C})=\mathcal{C}$ for a closed convex cone $\mathcal{C}\subseteq V$ which spans $V$ and contains no line.
Suppose $f(x)=qx$ for some $x\in \mathcal{C}^\circ$ and $q>0$.
Let $R$ be an isolated extremal ray of $\mathcal{C}$.
Then replacing $f$ by a positive power, $f(R)=R$ and $f|_R=q\,\id$.
\end{lemma}
\begin{proof}
Replacing $f$ by $f/q$, we may assume $q=1$.
Let $y\in R$ be of norm $1$.
For some $r>0$, $B(y,r)$ has no intersection with any extremal ray of $\mathcal{C}$ except $R$.
By \cite[Proposition 2.9]{MZ}, there exists a positive number $N$ such that $\frac{1}{N}<||f^n||<N$ for any $n\in \mathbb{Z}$.
So the set $\{f^n(y)\,|\, n\in \mathbb{Z}\}$ is bounded.
In particular, for some $a>b$, $|f^a(y)-f^b(y)|<\frac{r}{N}$.
Then $|f^{a-b}(y)-y|<N|f^a(y)-f^b(y)|<r$.
Note that the ray generated by $f^{a-b}(y)$ is also extremal in $\mathcal{C}$.
Then $f^{a-b}(y)\in R$.
By \cite[Proposition 2.9]{MZ}, all the eigenvalues of $f$ are of modulus $1$.
In particular, $f^{a-b}(y)=y$.
\end{proof}

The following is a slightly modified version of \cite[Lemma 2.7]{MZ}.
\begin{lemma}\label{lem-minface} Let $V$ be a positive dimensional real normed vector space.
Let $\mathcal{C}\subseteq V$ be a closed convex cone which spans $V$ and contains no line.
Let $x\in \mathcal{C}$ be a nonzero point.
Then there exists a unique minimal closed extremal face $F$ of $\mathcal{C}$ such that $x\in F$.
Furthermore, $x\in F^\circ$ (in the sense of the topology of the space spanned by $F$).
\end{lemma}
\begin{proof}
If $x\in \mathcal{C}^\circ$, then $F=\mathcal{C}$ and the lemma is trivial.
If $x\in \partial \mathcal{C}$, then \cite[Lemma 2.7]{MZ} proves that such $F$ exists and is contained in $\partial \mathcal{C}$.
If $x\not\in F^\circ$, then $x\in \partial F$ and we can apply \cite[Lemma 2.7]{MZ} again.
However, this contradicts the minimality of $F$.
\end{proof}

We recall \cite[Definition 4.1]{MZ_PG}.
Note that condition (3) is satisfied in the log minimal model program.
\begin{definition}\label{def:extrem_ray} Let $X$ be a projective variety.
Let $C$ be a curve such that $R_C$ is an extremal ray in $\NE(X)$.
We say $C$ or $R_C$ is {\it contractible} if there is a surjective morphism
$\pi : X \to Y$ to a projective variety $Y$ such that
the following hold.

\begin{itemize}
\item[(1)] $\pi_*\mathcal{O}_X=\mathcal{O}_Y$.
\item[(2)] Let $C'$ be a curve in $X$.
Then $\pi(C')$ is a point if and only if $[C'] \in R_C$.
\item[(3)] Let $D$ be a $\Q$-Cartier divisor of $X$.
Then $D\cdot C=0$ if and only if $D\equiv \pi^*D_Y$ (numerical equivalence) for some $\Q$-Cartier divisor $D_Y$ of $Y$.
\end{itemize}
\end{definition}

\begin{theorem}\label{thm-cont}
Let $f:X\to X$ be a surjective endomorphism of a normal projective variety $X$ with $(X,\Delta)$ being klt for some effective $\Q$-divisor $\Delta$.
Suppose $f_*x\equiv_w \lambda x$ for some $\lambda>0$ and nonzero $x\in \NE(X)$ with $(K_X+\Delta)\cdot x\le 0$.
Then one of the following holds.
\begin{itemize}
\item[(1)] $D\cdot x\ge 0$ for any effective Cartier divisor $D$.
\item[(2)] Replacing $f$ by a positive power, $\lambda$ is a positive integer and $f_*C\equiv_w \lambda C$ for some rational curve $C$ with $R_C$ being a contractible extremal ray of $\NE(X)$.
\end{itemize}
\end{theorem}
\begin{proof}
By Lemma \ref{lem-minface}, there exists a unique minimal closed extremal face $F$ of $\NE(X)$ containing $x$. By the uniqueness, we have $f_*(F)=F$.
Suppose $D\cdot x<0$ for some effective Cartier divisor $D$.
Since $(X,\Delta)$ is klt, $(X,\Delta+\epsilon D)$ is klt for some $\epsilon>0$ (cf.~\cite[Corollary 2.35]{KM}).
Note that $(K_X+\Delta+\epsilon D)\cdot x<0$.
By the cone theorem (cf.~\cite[Theorem 3.7]{KM}), $F$ contains some $(K_X+\Delta+\epsilon D)$-negative contractible extremal ray $R_C$ which is isolated.
Note that $x\in F^\circ$.
By Lemma \ref{lem-fix-ray}, after replacing $f$ by a positive power, $f_*C\equiv_w \lambda C$.
Let $A$ be any ample Cartier divisor.
We have that $A\cdot f_*C=\lambda A\cdot C$ is a positive integer and hence $\lambda$ is a rational number.
Since $\lambda$ is also an algebraic integer, $\lambda$ is an integer.
\end{proof}

%
%

\begin{lemma}\label{lem-cone-j2} Let $f:V\to V$ be an invertible linear map of a positive dimensional real normed vector space $V$ such that $f(C)=C$ for a closed convex cone $C\subseteq V$ which contains no line.
Suppose $f(x)=x$ for some nonzero $x\in V$ and $f(y)-y=ax$ for some $y\in C$ and real number $a$.
Then $a=0$.
\end{lemma}
\begin{proof}
Suppose $a>0$.
Then $\lim\limits_{n\to +\infty} \frac{f^n(y)}{|f^n(y)|}= \frac{x}{|x|}\in C$ and
$\lim\limits_{n\to -\infty} \frac{f^n(y)}{|f^n(y)|}= \frac{-x}{|x|}\in C$.
Since $x\neq 0$, this contradicts that $C$ contains no line.
The case $a<0$ is similar.
\end{proof}

Now we prove Theorem \ref{main-thm-qa-a}.
\begin{theorem}\label{thm-qa-a}
Let $f:X\to X$ be a surjective endomorphism of a normal projective variety $X$ such that $B=f^*D-D$ is big for some Cartier divisor $D$.
Suppose $(X,\Delta)$ is klt and $K_X+\Delta\equiv 0$ for some effective $\Q$-divisor $\Delta$.
Then replacing $f$ by a positive power, there is an $f$-equivariant sequence of contractions of extremal rays
$$X=X_1\to \cdots \to X_i \to \cdots \to X_r$$
(i.e. $f=f_1$ descends to $f_i$ on each $X_i$), such that we have:
\begin{itemize}
\item[(1)] For each $i\le r$, $(X_i, \Delta_i)$ is klt and $K_{X_i}+\Delta_i\sim_{\mathbb{Q}} 0$ for some effective $\mathbb{Q}$-divisor $\Delta_i$.
\item[(2)] For each $i<r$, the $i$-th contraction $\pi_i:X_i\to X_{i+1}$  of the extremal ray $R_{C_i}$ is birational and $(f_i)_*C_i\equiv_w C_i$.
\item[(3)] For each $i\le r$, there exists big Cartier divisor $B_i$ of $X_i$ such that $B_1=B$ and $B_i=\pi_i^*B_{i+1}$ for $i<r$.
\item[(4)] For each $i\le r$, if $(f_i)_*x_i\equiv_w x_i$ for some $x_i\in \NE(X_i)$, then $B_i\cdot x_i=0$.
\item[(5)] $f_r$ is amplified. For each $i\le r$, $f_i$ is of positive entropy and quasi-amplified and $\Per(f_i)\cap U_i$ is countable and Zariski dense in $X_i$ for some open subset $U_i\subseteq X_i$. 
\item[(6)] Suppose the base field $k$ is uncountable. For each $i\le r$, $f_i$ has a Zariski dense orbit.
\end{itemize}
\end{theorem}

\begin{proof}
If $\dim(X)=1$, then $B$ is ample and the theorem is trivial by taking $r=1$.
Suppose $\dim(X)>1$.
Let $r$ be the maximal integer such that we have an $f$-equivariant sequence of contractions of extremal rays
$$X=X_1\to \cdots \to X_i \to \cdots \to X_r,$$ 
satisfying (1) - (4).
Note that for each step, $\rho(X_i)=\rho(X_{i+1})+1$.
Then $r\le \rho(X)$.
When $r=1$, (2) and (3) are automatically true, (1) is satisfied by taking $\Delta_1=\Delta$ and applying \cite[Chapter V, Corollary 4.9]{ZDA}, and (4) is satisfied by the projection formula.

Suppose $f_r$ is amplified. Then we stop and (5) follows from Lemma \ref{lem-lift-des} and Theorem \ref{thm-amp-pcd}, and (6) follows from Theorem \ref{thm-amp-orbit}.

Suppose $f_r$ is not amplified.
There exists some $x_r\in \NE(X_r)\backslash\{0\}$ such that $(f_r)_*x_r\equiv_w x_r$ by Proposition \ref{prop-a-cri}.
Since (4) holds true for $i\le r$, $B_r\cdot x_r=0$.
By Theorem \ref{thm-cont}, after replacing $f$ by some positive power, we may assume $x_r=C_r$ for some contractible rational curve $C_r$.
Let $\pi_r:X_r\to X_{r+1}$ be the induced $f_r$-equivariant contraction.
By the cone theorem (cf. \cite[Theorem 3.7(4)]{KM}), $B_r=\pi_r^*B_{r+1}$ for some Cartier divisor $B_{r+1}$ of $X_{r+1}$.
Since $B_r$ is big, $\pi_r$ is birational and $B_{r+1}$ is big too.
Let $\Delta_{r+1}=(\pi_r)_*\Delta_r$.
Then $K_{X_{r+1}}+\Delta_{r+1}=(\pi_r)_*(K_{X_r}+\Delta_r)\sim_{\mathbb{Q}}0$
and hence $K_{X_r}+\Delta_r=(\pi_r)^*(K_{X_{r+1}}+\Delta_{r+1})$.
So $(X_{r+1},\Delta_{r+1})$ is klt.
Suppose $(f_{r+1})_*x_{r+1}\equiv_w x_{r+1}$ for some $x_{r+1}\in \NE(X_{r+1})$.
Since $\pi_r$ is surjective, there exists some $y_r\in \NE(X_r)$ such that $(\pi_r)_*y_r\equiv_w x_{r+1}$.
Note that $(\pi_r)_*((f_r)_*y_r-y_r)\equiv_w (f_{r+1})_*x_{r+1}-x_{r+1}\equiv_w 0$.
Then $(f_r)_*y_r-y_r\equiv_w mC_r$ for some real number $m$.
By Lemma \ref{lem-cone-j2}, $m=0$.
By the projection formula, $B_{r+1}\cdot x_{r+1}=B_r\cdot y_r=0$ since (4) holds for $i\le r$.
Now we have a longer sequence and we have checked that (1) - (4) hold for $r+1$.
However, this contradicts the maximality of $r$.
\end{proof}

\section{The Hyperk\"ahler case and proof of Theorem \ref{main-thm-hk}}
In this section, we always work over $\mathbb{C}$.
Let $X$ be a projective Hyperk\"ahler manifold.
There is a Beauville-Bogomolov-Fujiki's form $q$ on $\N^1(X)$ with signature $(1, 0, \rho(X)-1)$ where $\rho(X)$ is the Picard number of $X$ (cf.~\cite{Huy99}).
\begin{lemma}\label{lem-hk-sum-big}
Let $X$ be a projective Hyperk\"ahler manifold.
Let $D_1$ and $D_2$ be nef $\R$-Cartier divisors which are linearly independent in $\N^1(X)$.
Then $q(D_1,D_2)>0$ and $D_1+D_2$ is nef and big.
\end{lemma}
\begin{proof}
Note that $q(D_1):=q(D_1,D_1)\ge 0$ and $q(D_2)\ge 0$.
Since $D_1$ and $D_2$ are linearly independent, $q(D_1, D_2)>0$ by observing the signature.
Then $q(D_1+D_2)=q(D_1)+q(D_2)+2q(D_1,D_2)>0$.
By the claim in \cite[Proposition 26.13]{GHJ}, $D_1+D_2$ is big.
\end{proof}

\begin{lemma}\label{lem-HK-finite}
Let $f:X\to X$ be an automorphism of a  projective Hyperk\"ahler manifold $X$.
Suppose $f^*D\equiv D$ for some $\R$-Cartier divisor $D$ such that $D$ is big or $q(D)>0$.
Then $f$ has finite order.
\end{lemma}
\begin{proof}
Note that $q(D)>0$ implies either $D$ or $-D$ is big.
Then we may assume $D$ is big.
Applying \cite[Proposition 2.9]{MZ} and Kronecker's theorem, we have $f^*|_{\N^1(X)}=\id$ after replacing $f$ by a positive power.
Then $f$ has finite order by \cite[Corollary 2.7]{Ogu08} (cf.~\cite[Section 9]{Huy99}). 
\end{proof}

We recall \cite[Lemma 2.8]{Ogu07} and provide a simplified proof in our situation.
We also refer to \cite{HKZ} for related results of birational automorphisms group of null entropy. 

\begin{proposition}\label{prop-hk-null}
Let $f:X\to X$ be an automorphism of a projective Hyperk\"ahler manifold $X$.
Then the following are equivalent.
\begin{itemize}
\item[(1)] $f$ is of null entropy.
\item[(2)] $f^*D\equiv D$ for some nef $\mathbb{R}$-Cartier divisor $D\not\equiv 0$.
\end{itemize}
Suppose further the order of $f$ is infinite.
Then the above are equivalent to
\begin{itemize}
\item[(3)] There is a unique (up to scalar) nef Cartier divisor $D\not\equiv 0$ such that $f^*D\sim D$.
\end{itemize}
\end{proposition}
\begin{proof}
(1) implies (2) by the Perron-Frobenius theorem.
Suppose $f$ is of positive entropy and (2) holds.
Then $f^*D'\equiv rD'$ for some nonzero $D'\in \Nef(X)$ where $r>1$ is the spectral radius of $f^*|_{\N^1(X)}$.
Note that $q(D,D')=q(f^*D, f^*D')=rq(D,D')$.
Then $q(D,D')=0$.
However, $D$ and $D'$ are linearly independent in $\N^1(X)$.
By Lemma \ref{lem-hk-sum-big}, $q(D,D')>0$, a contradiction.
So (2) implies (1).

Suppose now that $f$ has infinite order.
Clearly, (3) implies (2).
Suppose $f^*D\equiv D$ and $f^*D'\equiv D'$ for two linearly independent $D, D'\in \Nef(X)$.
by Lemma \ref{lem-hk-sum-big}, $f^*(D+D')\equiv D+D'$ and $D+D'$ is big, a contradiction by Lemma \ref{lem-HK-finite}.
Suppose all the eigenvalues of $(f^n)^*|_{\N^1(X)}$ are $1$ for some $n>0$.
Let $A$ be an ample Cartier divisor.
Then $\lim\limits_{i\to +\infty} R_{(f^{in})^*A}=R_D$ for some nef Cartier divisor $D$ and $(f^n)^*D\equiv D$.
Let $D'=\sum\limits_{i=0}^{n-1} (f^i)^*D$.
Then $D'$ is nef and Cartier and $f^*D'\equiv D'$.
Since $q(X)=0$, $f^*D'\sim D'$ after replacing $D'$ by $mD'$ for some integer $m>0$.
So (1) implies (3).
\end{proof}

\begin{proof}[Proof of Theorem \ref{main-thm-hk}] 
(1) and (2) are equivalent by Proposition \ref{prop-hk-null}.
(3) implies (4) by Theorem \ref{thm-qa-a}. (4) implies (1) by Lemmas \ref{lem-amp-pos} and \ref{lem-lift-des}.

Suppose $f$ is of positive entropy. By the Perron-Frobenius theorem, $f^*D_1\equiv aD_1$ for some nef $\mathbb{R}$-Cartier divisor $D_1$ and $a>1$, and $f^*D_2\equiv bD_2$ for some $\mathbb{R}$-Cartier divisor $D_2$ and $b<1$ (Indeed $ab=1$).
Note that $D_1$ and $D_2$ are linearly independent in $\N^1(X)$.
Let $D=D_1-D_2$.
Then $f^*D-D=(a-1)D_1+(1-b)D_2$ is nef and big by Lemma \ref{lem-hk-sum-big}.
In particular, (1) implies (3).

The last assertion follows from Lemma \ref{lem-pcd-pos}.
\end{proof}

\begin{remark}
If Question \ref{que-pos} has a positive answer, then the above equivalent conditions are also equivalent to that ``$f^n$ is birationally equivalent to some PCD automorphism for some $n>0$''.
\end{remark}

For the K3 surfaces, the following result is well-known to experts. For the convenience of the readers, we provide a proof here.
\begin{corollary}\label{cor-k3}
Let $f:X\to X$ be an automorphism of a projective K3 surface $X$.
Then the following are equivalent.
\begin{itemize}
\item[(1)] $f$ is of positive entropy.
\item[(2)] $\Per(f)\cap U$ is countable and Zariski dense for some open dense subset $U$ of $X$.
\item[(3)] $f$ has a Zariski dense orbit.
\end{itemize}
\end{corollary}
\begin{proof}
(1) implies (2) by Theorems \ref{main-thm-hk} and \ref{main-thm-qa-a}.
Let $x\in X$ be any point and let $Z$ be the closure of the orbit $\{f^n(x)\,|\,n\ge 0\}$.
Then $f(Z)\subseteq Z$ implies $f(Z)=Z$ and hence $Z$ is also the closure of the set $\{f^n(x)\,|\,n\in\mathbb{Z}\}$.
Then (1) and (3) are equivalent by \cite[Theorem 1.4]{Ogu07}.

Suppose $f$ is of null entropy and (2) holds. By Proposition \ref{prop-hk-null}, $f^*D\equiv D$ for some nef Cartier divisor $D\not\equiv 0$.
If $D^2>0$, then $D$ is nef and big and hence $f$ has finite order by Lemma \ref{lem-HK-finite}.
In particular, $\Per(f)\cap U=U$ is uncountable for any open dense subset $U$ of $X$, a contradiction.
If $D^2=0$, the Riemann-Roch theorem implies that $D$ is basepoint free.
Then we have an $f$-equivariant elliptic fibration $\pi:X\to \mathbb{P}^1$.
Denote by $g:=f|_{\mathbb{P}^1}$.
By (2), $\Per(g)$ is Zariski dense in $\mathbb{P}^1$ and hence $g$ has finite order.
Replacing $f$ by a positive power, we may assume $g=\id$.
Let $y$ be a general point of $\mathbb{P}^1$ such that the fibre $X_y:=\pi^{-1}(y)$ is a smooth elliptic curve and $\Per(f)\cap U\cap X_y\neq \emptyset$.
Then we may assume $f|_{X_y}$ is an isogeny after replacing $f$ by a positive power.
It is known that an (algebraic group) automorphism of an elliptic curve has finite order.
So $\Per(f|_{X_y})=X_y$ and hence $\Per(f)\cap U$ is uncountable, a contradiction.
\end{proof}

\begin{remark}
In the above corollary, Cantat \cite[2]{Can} showed that (1) and (3) are equivalent even when $X$ is not necessarily projective; see also \cite[Theorem 1.4]{Ogu07}. On the other hand, Xie  \cite[Theorem 1.1]{Xie} showed that (1) implies (2) even when $f$ is only a birational automorphism of a smooth projective surface over an algebraically closed field $k$ with $char\, k\neq 2, 3$.

In general, Amerik and Campana \cite{AC} showed that for a dominant meromorphic endomorphism $f:X\dasharrow X$ of a compact K\"ahler manifold $X$, there is a dominant meromorphic map $\pi:X\dasharrow Y$ onto a compact  K\"ahler manifold $Y$, such that
$\pi\circ f=\pi$ and the general fibre $X_y$ of $\pi$ is the Zariski closure of the orbit by $f$ of a general point of $X_y$.
Applying this, Lo Bianco \cite[Main Theorem]{Lo} showed that (1) implies (3) for the Hyperk\"ahler manifolds; see also Theorem \ref{main-thm-qa-a} for another application.
\end{remark}

K.~Oguiso suggested the following example that some K3 surface admits an automorphism of positive entropy which is not PCD.
\begin{example}[Oguiso]\label{exa-k3-1}
Let $S = {\rm Km} (E \times F)$ be the Kummer surface associated to the product of two non-isogenous elliptic curves $E$ and $F$.
Dinh and Oguiso \cite[Proposition 3.7]{DO} showed that there exists a subgroup $G$ of $\Aut(S)$ such that $G$ is not finitely generated.
Then $G$ contains some automorphism $f$ of positive entropy (cf.~\cite[Proposition 1.3]{HKZ}).
Note that $\Per(f)$ contains at least 8 curves (cf.~\cite[Section 3, Figure 1]{DO}).
So $f$ is not PCD.
\end{example}

D.-Q. Zhang suggested the following example that some K3 surface admits a PCD automorphism which is not amplified.

\begin{example}[Zhang]\label{exa-k3-2}
Let $S = {\rm Km} (E \times E)$ be the Kummer surface associated to the square of an elliptic curve $E$.
Let $\Tor_2$ be the set of 2-torsion points of $E\times E$. 
Denote by $\pi: \widetilde{S}\to E\times E$ the blowup of $\Tor_2$.
Denote by $\tau:\widetilde{S}\to S$ be the finite surjective morphism of degree 2.
Let $f:E\times E\to E\times E$ be the automorphism defined by $f(a,b)=(5a+8b,8a+13b)$.
Then $f$ fixes all the 2-torison points and $f$ is PCD (cf.~Theorem \ref{thm-pcd-cri}).
For any $n>0$, $(f^n)_*|_{T_P}$ is not a scalar action where $P$ is a 2-torsion point and $T_P$ is the tangent space of $E\times E$ at $P$.
Denote by $\widetilde{f}$ the equivariant lifting of $f$ to $\widetilde{S}$ and $f_S:=f|_S$ the equivariant descending of $f$ to $S$.
Then $\Per(\widetilde{f})=\pi^{-1}(\Per(f)\backslash\Tor_2)\bigcup \Per(\widetilde{f}|_{\pi^{-1}(\Tor_2)})$
is also countable.
In particular, $\widetilde{f}$ is PCD and hence so is $f_S$ by Lemma \ref{lem-fin-lift-des}.
Note that the 16 $\pi$-exceptional divisors are $\widetilde{f}$-invariant.
So $\widetilde{f}$ is not amplified (cf.~Proposition \ref{prop-a-cri}).
Similarly, $f$ is not amplified.
\end{example}

At the end of this section, we would like to ask a related question.
\begin{question}\label{que-hk}
Let $f$ be an automorphism of a projective Hyperk\"ahler manifold $X$.
Suppose $\Per(f)\cap U$ is countable and Zariski dense for some open dense subset $U$ of $X$. Will $f$ be of positive entropy?
Does $f$ admit a Zariski dense orbit?
\end{question}

\section{Case of abelian varieties}

Let $A$ be an abelian variety of dimension $g$.
We recall some facts from \cite[Sections 6, 8, 16]{Mum}.
Let $n$ be a nonzero integer. 
Denote by $n_A:A\to A$ the isogeny sending $a$ to $na$.
Let $L$ be a Cartier divisor of $A$.
Then the Euler characteristic $\chi(L)$ equals to $\frac{L^g}{g!}$ where $L^g$ is the highest self intersection of $L$.
Consider the following homomorphism to the dual abelian variety
$$\begin{aligned}
\phi_L: \, A \, \to \, A^{\vee}:=\Pic^0(A) \\
a \, \mapsto \, T_a^*L - L
\end{aligned}$$
where $T_a$ is the translation map by $a$.
Denote by $K(L)$ the kernel of $\phi_L$.
For any connected closed subgroup $B$, $L|_B\equiv 0$ if and only if $B$ is a subgroup of $K(L)$. 
In particular, $L|_{K(L)}\equiv 0$.
If $L$ is ample, then $K(L)$ is finite and hence $\phi_L$ is an isogeny.
If $K(L)$ is finite, then $\chi(L)\neq 0$.
For any surjective endomorphism $f:A\to A$,
we have $\phi_{f^*L}=f^{\vee}\circ \phi_L\circ f$, where $f^{\vee}:A^{\vee}\to A^{\vee}$ is the dual map of $f$.

In the following, we show that the building blocks of surjective endomorphisms of abelian varieties are automorphisms and amplified endomorphisms.
\begin{proposition}\label{prop-ab-end-des}
Let $f:A\to A$ be a surjective endomorphism of an abelian variety $A$ of positive dimension.
Then there is an $f$-equivariant surjective homomorphism $\pi:A\to B$ to an abelian variety $B$ of positive dimension such that the induced $f|_B$ is either an  automorphism or an amplified endomorphism.
\end{proposition}
\begin{proof}
We show by induction on $n:=\dim(A)$.
Write $f=g+a$ where $g$ is an isogeny and $a\in A$.
If $\dim(A)=1$, then $f$ is either an automorphism or a polarized endomorphism.
Suppose $f$ is neither amplified nor an automorphism.
Then so is $g$ since $T_a^*|_{\N^1(A)}=\id$.
Hence, $g^*L\equiv L$ for some Cartier divisor $L\not\equiv 0$.
Since $\deg g>1$, $L^n=(g^*L)^n=(\deg g)L^n$ implies that $L^n=0$.
Then $0<\dim(K(L))<n$.
Note that $\phi_L=\phi_{f^*L}=f^{\vee}\circ \phi_L\circ f$.
Let $Z$ be the neutral component of $K(L)$.
Since $0\in g(Z)$ and $L|_{g(Z)}\equiv 0$, we have $g(Z)\le K(L)$ and hence $g(Z)=Z$.
Let $B=A/Z$ and define $h:B\to B$ via $h(\overline{x})=\overline{f(x)}$.
It is easy to check that $h$ is well defined.
Note that $0<\dim(B)<\dim(A)$.
Then we are done by induction.
\end{proof}

Let $A$ be an abelian variety of dimension $n$.
Denote by $H^{k,k}(A, \mathbb{R})=H^{k,k}(A, \mathbb{C})\cap H^{2k}(A,\mathbb{R})$.
For $k=1$ and $n-1$, denote by $\Pos^k(A)$ the cone of positive $(k, k)$-forms in $H^{k,k}(A, \mathbb{R})$.
Note that $\Pos^1(A)\cap \N^1(A)=\Nef(A)=\PE^1(A)$ and $\Pos^{n-1}(A)\cap \N_1(A)=\NE(A)$.
We refer to \cite{DELV} and \cite[Chapter III]{Dem} for the details.

\begin{theorem}\label{thm-ab-amp-cri}
Let $f:A\to A$ be a surjective endomorphism of an abelian variety $A$ of dimension $n$.
Then the following are equivalent.
\begin{itemize}
\item[(1)] $f$ is amplified.
\item[(2)] $f^*\omega-\omega\in \Pos^1(A)^{\circ}$ for some $\omega\in H^{1,1}(A,\mathbb{R})$.
\item[(3)] No eigenvalue of $f^*|_{H^1(X,\mathcal{O}_X)}$ is of modulus $1$.
\item[(4)] $f_*Z-Z\in \NE(A)^{\circ}$ for some $Z\in \N_1(A)$.
\item[(5)] For any $Z\in \NE(X)$, $f_*Z\equiv_w Z$ implies $Z\equiv_w 0$.
\item[(6)] For any $D\in \Nef(X)$, $f^*D\equiv D$ implies $D\equiv 0$.
\item[(7)] For any $\omega\in \Pos^{n-1}(A)$, $f^*\omega=(\deg f) \omega$ implies $\omega=0$.
\item[(8)] For any $\omega\in \Pos^1(A)$, $f^*\omega=\omega$ implies $\omega=0$.
\end{itemize}
\end{theorem}
\begin{proof}
(1) and (5) are equivalent by Proposition \ref{prop-a-cri}.
(4) and (6) are equivalent by the same proof of Proposition \ref{prop-a-cri}.
Clearly, (1) implies (2).

Consider the Jordan canonical form of $f^*|_{H^1(X,\mathcal{O}_X)}$ with the Jordan blocks $J_1,\cdots, J_m$.
Let $r_i$ be the rank of $J_i$ and $\lambda_i$ the corresponding eigenvalue of $J_i$. 
Let $\{x_{i_j}\}_j$ be the corresponding basis of $J_i$ such that $f^*x_{i_j}=\lambda_i x_{i_j}+x_{i_{j+1}}$ if $j<r_i$ and $f^*x_{i_j}=\lambda_i x_{i_j}$ if $j=r_i$.
Note that $\{x_{i_j}\wedge \overline{x_{i'_{j'}}}\}_{i_j, i'_{j'}}$ forms a basis of $H^{1,1}(A,\mathbb{C})$ and 
$f^*|_{H^1(X,\mathcal{O}_X)}$ determines $f^*|_{H^{1,1}(A,\mathbb{C})}$.
Suppose $|\lambda_1|=1$.
If $r_1>1$, then $f^*(x_{1_1}\wedge \overline{x_{1_1}})-x_{1_1}\wedge \overline{x_{1_1}}=\lambda_1 x_{1_1}\wedge \overline{x_{1_2}}+\overline{\lambda_1} x_{1_2}\wedge \overline{x_{1_1}}+x_{1_2}\wedge \overline{x_{1_2}}$.
If $r_1=1$, then $f^*(x_{1_1}\wedge \overline{x_{1_1}})-x_{1_1}\wedge \overline{x_{1_1}}=0$.
Note that the coefficient of $x_{1_1}\wedge \overline{x_{1_1}}$ in $f^*(x_{i_j}\wedge \overline{x_{i'_{j'}}})$ is $0$ for any $i_j\neq 1_1$ or $i'_{j'}\neq 1_1$.
Therefore, for any $\omega\in H^{1,1}(A,\mathbb{R})$, the coefficient of $x_{1_1}\wedge \overline{x_{1_1}}$ in $f^*\omega-\omega$ is $0$ and hence $f^*\omega-\omega\not\in \Pos^1(A)^{\circ}$.
So (2) implies (3).

Suppose $f^*\omega=\omega$ for some nonzero $\omega\in \Pos^1(A)$.
Write $\omega=\sum a_{i_j,i'_{j'}} x_{i_j}\wedge \overline{x_{i'_{j'}}}$.
Let $s$ be the minimal one such that $a_{s_j,s_j}\neq 0$ for some $j$.
Let $t$ be the minimal one such that $a_{s_t,s_j}\neq 0$ for some $j$.
Let $t'$ be the minimal one such that $a_{s_t,s_{t'}}\neq 0$.
Then the coefficient of $x_{s_t}\wedge\overline{x_{s_{t'}}}$ in $f^*\omega$ is $|\lambda_s|^2a_{s_t,s_{t'}}$.
So $|\lambda_s|^2=1$ and hence (3) implies (8).
Clearly, (8) implies (3) and (6).

Note that $f^*(x_{1_1}\wedge\cdots\wedge x_{m_{r_m}} \wedge \overline{x_{1_1}}\wedge\cdots\wedge \overline{x_{m_{r_m}}})=|\lambda_1^{r_1}\cdots \lambda_m^{r_m}|^2x_{1_1}\wedge\cdots\wedge x_{m_{r_m}} \wedge \overline{x_{1_1}}\wedge\cdots\wedge \overline{x_{m_{r_m}}}$.
Then $\deg f=|\lambda_1^{r_1}\cdots \lambda_m^{r_m}|^2$.
Denote by the $(n-1, n-1)$ form $$x^*_{i_j}\wedge \overline{x^*_{i'_{j'}}}=x_{1_1}\wedge\cdots\wedge \widehat{x_{i_j}}\wedge \cdots \wedge x_{m_{r_m}} \wedge \overline{x_{1_1}}\wedge\cdots\wedge\widehat{\overline{x_{i'_{j'}}}}\wedge\cdots\wedge \overline{x_{m_{r_m}}}.$$
Suppose $f^*\omega=(\deg f)\omega$ for some nonzero $\omega\in \Pos^{n-1}(A)$.
Write $\omega=\sum a_{i_j,i'_{j'}} x^*_{i_j}\wedge \overline{x^*_{i'_{j'}}}$.
Take $s, t, t'$ like above.
Then the coefficient of $x^*_{s_t}\wedge\overline{x^*_{s_{t'}}}$ in $f^*\omega$ is $\frac{\deg f}{|\lambda_s|^2}a_{s_t,s_{t'}}$.
So $|\lambda_s|^2=1$ and hence (3) implies (7).

Suppose $f_*Z\equiv_w Z$ for some nonzero $Z\in \NE(A)$.
By the projection formula, $f^*Z\equiv_w (\deg f) Z$.
So (7) implies (5).
\end{proof}

\begin{lemma}\label{lem-iso-pd}
Let $f:A\to A$ be an isogeny of an abelian variety $A$.
Then $\Per(f)$ is Zariski dense in $A$.
\end{lemma}
\begin{proof}
Let $n$ be a positive integer such that $(n,\deg f)=1$.
Denote by $\Tor_n(A)$ be the set of $n$-torsion points of $A$.
For any $x\in \Tor_n(A)$, $f(x)\in \Tor_n(A)$.
We claim that $f|_{\Tor_n(A)}$ is a bijection.
First, if $f(x)=0$, then $(\deg f)x=0$.
Since $(n, \deg f)=1$, $x=0$.
So $f|_{\Tor_n(A)}$ is injective and hence bijective since $\Tor_n(A)$ is finite.
In particular, $\bigcup_{(n,\deg f)=1} \Tor_n(A)\subseteq \Per(f)$.
Let $B$ be the closure of $\bigcup_{(n,\deg f)=1} \Tor_n(A)$.
Then $B$ is a closed subgroup of $A$.
Suppose $B\neq A$.
By Poincar\'e's complete reducibility theorem (cf.~\cite[\S 19, Theorem 1]{Mum}), there is an abelian subvariety $C$ of $A$ such that $B\cap C$ is finite and $B+C=A$.
Then for any $n>\sharp B\cap C$ and $(n, \deg f)=1$, we have  $\{0\}\neq \Tor_n(C)\subseteq \Tor_n(A)\subseteq B$ and hence $\Tor_n(C)\subseteq B\cap C$. 
However, $\sharp \Tor_n(C)>n>\sharp B\cap C$,  a contradiction.
\end{proof}

\begin{lemma}\label{lem-pcd+a}
Let $f:A\to A$ be a PCD surjective endomorphism of an abelian variety $A$.
Then $f+a$ is PCD for any $a\in A$.
\end{lemma}
\begin{proof}
Denote by $g:=f+a$.
Since $f$ is PCD, $\Fix(f^n)\neq \emptyset$ for some $n>0$ and we may assume $f^n$ is an isogeny.
Note that $g^n=f^n+b$ for some $b\in A$.
By Proposition \ref{prop-ext}, $\Fix(f^{in})$ is finite for each $i>0$.
So $f^{in}-\id_A$ is still an isogeny and hence $\Fix(g^{in})$ is finite and nonempty for each $i>0$.
Since $\Fix(g^n)\neq \emptyset$, $g^n$ is an isogeny after choosing a suitable origin point.
By Lemma \ref{lem-iso-pd}, $\Per(g^n)$ is Zariski dense and hence $g$ is PCD by Proposition \ref{prop-ext}.
\end{proof}

Krieger and Reschke \cite[Proposition 2.5]{KR} gave the following characterization of PCD isogenies. By applying the above lemma, we generalize it a little bit.
\begin{theorem}\label{thm-pcd-cri}
Let $f:A\to A$ be a surjective endomorphism of an abelian variety $A$ of dimension $n$.
Then  $f$ is PCD if and only if none of the eigenvalues of $f^*|_{H^1(X,\mathcal{O}_X)}$ are roots of unity.
\end{theorem}
\begin{proof}
Write $f=g+a$ where $g$ is an isogeny and $a\in A$.
By Lemma \ref{lem-pcd+a}, $f$ is PCD if and only if so is $g$.
Note that $T_a^*|_{H^1(X,\mathcal{O}_X)}=\id$.
So the theorem follows from \cite[Proposition 2.5]{KR}.
\end{proof}

Now we can construct a PCD endomorphism (automorphism) which is not amplified.
\begin{example}\label{exa-salem}
Let $\varphi(x)=\sum\limits_{i=0}^n a_i x^i$ be a Salem polynomial where $a_0=a_n=1$ and $n>2$.
For example, we may take the Lehmer's polynomial $$\varphi(x)=x^{10} +x^9-x^7-x^6-x^5-x^4-x^3+x+1.$$
It is known that $\varphi(x)$ is irreducible and it has exactly two real roots $\alpha>1$ and $1/\alpha$ off the unit circle
$S^1 :=\{z\in\mathbb{C}\,|\,|z|=1\}$.
Note that no root of $\varphi$ is a root of unity.
Let $M\in \GL_n(\mathbb{Z})$ such that the characteristic polynomial of $M$ is $\varphi$.
For example, we may take $M$ as
$$\begin{pmatrix} 
0    &0    &\cdots    &0    &-a_0    \\ 
1    &0    &\cdots    &0    &-a_1  \\ 
0    &1    &\cdots    &0    &-a_2\\
\vdots & \vdots &\ddots &\vdots &\vdots\\
0    &0    &\cdots    &1    &-a_n
\end{pmatrix}$$
Let $E$ be an elliptic curve and $A:=E^{\times n}$.
Then $M$ induces an automorphism $f:A\to A$ via $f(x)=Mx$.
Note that no eigenvalue of $f^*|_{H^1(A,\mathcal{O}_A)}$ is a root of unity.
By Theorem \ref{thm-pcd-cri}, $f$ is PCD.
However, some eigenvalue is of modulus $1$.
So $f$ is not amplified by Theorem \ref{thm-ab-amp-cri}.
\end{example}

Applying \cite[Proposition 2.5]{KR} (cf.~Theorem \ref{thm-pcd-cri}), we can rewrite a result of Pink and Roessler.
\begin{theorem}\label{thm-pcd-PR} (cf.~\cite[Theorem 2.4]{PR})
Let $f:A\to A$ be a PCD endomorphism of an abelian variety $A$.
Suppose $f(X)=X$ for some (irreducible) closed subvariety $X$ of $A$.
Then $X$ is an abelian variety.
\end{theorem}

Next, we consider the restriction, lifting and descending problems.
\begin{lemma}\label{lem-pcd-res}
Let $f:A\to A$ be a PCD surjective endomorphism of an abelian variety $A$.
Let $B$ be an (irreducible) closed subvariety of $A$ such that $f(B)=B$.
Then $f|_B$ is also PCD.
\end{lemma}
\begin{proof}
By Theorem \ref{thm-pcd-PR}, $B$ is an abelian variety and we may assume that $B$ is a subgroup of $A$.
Denote by $t:=f(0)\in B$ and $g:=f-t$.
Then $g(B)=f(B)-t=B-t=B$ and $g(0)=0$.
Since $f$ is PCD, so is $g$ by Lemma \ref{lem-pcd+a}. 
By Proposition \ref{prop-ext}, $\Fix(g^i|_B)$ is finite for each $i>0$ and hence $g|_B$ is PCD by Lemma \ref{lem-iso-pd}. 
So $f|_B$ is PCD by Lemma \ref{lem-pcd+a} again.
\end{proof}

\begin{lemma}\label{lem-pcd-des}
Let $\pi:A\to B$ be a surjective morphism of abelian varieties.
Let $f:A\to A$ and $g:B\to B$ be surjective endomorphisms such that $\pi\circ f=g\circ \pi$.
Suppose $f$ is PCD.
Then so is $g$.
\end{lemma}
\begin{proof}
Replacing $f$ by a positive power, we may assume $f$ is an isogeny.
We may also assume $\pi$ is a homomorphism and hence $g$ is an isogeny.
By Proposition \ref{prop-ext}, we may work over an uncountable field.
It is clear that $\Per(g)$ is Zariski dense in $B$.
Suppose $\Fix(g^n)$ is infinite for some $n>0$.
Then $\Fix(g^n)$ is uncountable.
By Lemma \ref{lem-pcd-res}, for each $y\in \Fix(g^n)$, $f^n|_{A_y}$ is PCD where $A_y$ is an irreducible component of $\pi^{-1}(y)$.
In particular, $\Per(f^n|_{A_y})\neq \emptyset$ for each $y\in \Fix(g^n)$.
Note that $\Per(f)\supseteq \bigcup_{y\in \Fix(g^n)}\Per(f^n|_{A_y})$.
Then $\Per(f)$ is uncountable, a contradiction.
\end{proof}

\begin{lemma}\label{lem-ab-ns-res-surj}
Let $i:A\to B$ be an inclusion morphism of abelian varieties.
Then the restriction $i^*:\NS_{\Q}(B)\to \NS_{\Q}(A)$ is surjective.
\end{lemma}
\begin{proof}
We may assume $i$ is also a group homomorphism.
By Poincar\'e's complete reducibility theorem (cf.~\cite[\S 19, Theorem 1]{Mum}), there is an abelian subvariety $A'$ of $B$ such that $A\cap A'$ is finite and $A+A'=B$.
Define $\pi: A\times A'\to B$ via $\pi(a,a')=a+a'$.
Then $\pi$ is an isogeny.
Define $j:A\to A\times A'$ via $j(a)=(a,0)$.
Then $\pi\circ j=i$.
Note that $j^*|_{\NS_{\Q}(A\times A')}$ is surjective and $\pi^*|_{\NS_{\Q}(B)}$ is isomorphism.
Then $i^*|_{\NS_{\Q}(B)}$ is surjective.
\end{proof}

\begin{proposition}\label{prop-5lem} Consider the commutative diagram of abelian varieties
$$\xymatrix{
0\ar[r]&A\ar[r]^i\ar[d]^{f} &B\ar[r]^\pi\ar[d]^{g} &C\ar[r]\ar[d]^{h} &0\\
0\ar[r]&A\ar[r]^i &B\ar[r]^\pi &C\ar[r] &0
}$$
where $f, g, h$  are surjective endomorphisms.
Then $g$ is amplified (resp.~PCD) if and only if both are $f$ and $h$.
\end{proposition}
\begin{proof}
We may assume that $f, g, h$ are isogenies (cf.~Lemma \ref{lem-pcd+a}).

Suppose $g$ is amplified.
Clearly, $f$ is amplified.
Suppose $h$ is not amplified.
By Theorem \ref{thm-ab-amp-cri}, $h^*D\equiv D$ for some $D\in \Nef(C)\backslash\{0\}$.
Then $g^*\pi^*D\equiv \pi^*D$ and hence $g$ is not amplified by Theorem \ref{thm-ab-amp-cri} again.
So we get a contradiction.

Suppose $f$ and $h$ are amplified.
Let $V$ be the space of the image of $g^*|_{\NS_{\Q}(B)}-\id$.
By Lemma \ref{lem-ab-ns-res-surj}, $i^*D$ is ample for some $D\in V$, i.e., $D$ is $\pi$-ample.
Suppose $h^*E-E$ is ample for some $E\in \NS_{\Q}(C)$. 
Let $F:=g^*(\pi^*E)-\pi^*E=\pi^*(h^*E-E)\in V$.
By \cite[Proposition 1.45]{KM}, $nF+D\in V$ is ample for $n\gg 1$.

Suppose $g$ is PCD.
Then both are $f$ and $h$ by Lemmas \ref{lem-pcd-res} and \ref{lem-pcd-des}.
Suppose $f$ and $h$ are PCD and $g$ is not PCD.
By Proposition \ref{prop-ext} and Lemma \ref{lem-iso-pd}, $\Fix(g^n)$ is infinite for some $n>0$.
Let $Z$ be the neutral component of $\Fix(g^n)$.
Then $Z$ is an abelian variety of positive dimension and $\pi(Z)=0$.
Therefore, $Z\le i(A)$ and $i^{-1}(Z)\subseteq \Fix(f^n)$, a contradiction.
\end{proof}

\begin{proposition}\label{prop-dual}
Let $f:A\to A$ be a surjective endomorphism of an abelian variety $A$.
Let $f^{\vee}:A^{\vee}\to A^{\vee}$ be the dual endomorphism.
Then $f$ is amplified (resp.~PCD) if and only if so is $f^{\vee}$.
\end{proposition}
\begin{proof}
We may assume $f$ is an isogeny (cf.~Lemma \ref{lem-pcd+a}) and prove by induction on $\dim(A)$.
If $\dim(A)=0$, it is trivial.
If $\dim(A)=1$, $f$ is amplified (resp.~PCD) if and only if $\deg f>1$.
Note that $\deg f=\deg f^{\vee}$. So we are done.

Suppose $\dim(A)>1$. Since $f^{\vee\vee}=f$, it suffices for us to show that if $f^{\vee}$ is amplified (resp.~PCD), then so is $f$.
If $f^*|_{\N^1(A)}-\id$ is surjective, then $f$ is amplified (resp.~PCD) and there is nothing to prove.
Suppose now that $f^*L\equiv L$ for some Cartier divisor $L\not\equiv 0$.

Suppose $K(L)$ is finite.
Note that $\phi_L=f^{\vee}\circ \phi_L\circ f$.
If $H=(f^{\vee})^*D-D$ is ample, then $f^*(-\phi_L^*((f^{\vee})^*D))-(-\phi_L^*((f^{\vee})^*D))=\phi_L^*H$ is ample and hence $f$ is amplified.
If $f$ is not PCD, then $\Fix(f^n)$ is infinite for some $n>0$ by Proposition \ref{prop-ext} and Lemma \ref{lem-iso-pd}.
Let $Z$ be the neutral component of $\Fix(f^n)$ and denote by $i:Z\to A$ be the inclusion map.
Then $\dim(Z)>0$ and $f^n|_Z=\id_Z$.
Note that $i^{\vee}:A^{\vee}\to Z^{\vee}$ is surjective and $(f^\vee)^n|_{Z^{\vee}}=\id_Z^{\vee}=\id_{Z^{\vee}}$.
Then $f^{\vee}$ is not PCD by Lemma \ref{lem-pcd-des}.

Suppose $\dim(K(L))>0$.
Let $A_1$ be the neutral component of $K(L)$.
Then $\dim(A_1)>0$, $f(A_1)=A_1$ and we have the commutative diagram
$$\xymatrix{
0\ar[r]&A_1\ar[r]^i\ar[d]^{g} &A\ar[r]^\pi\ar[d]^{f} &A_2\ar[r]\ar[d]^{h} &0\\
0\ar[r]&A_1\ar[r]^i &A\ar[r]^\pi &A_2\ar[r] &0
}$$
where $A_2=A/A_1$ and $\dim(A_2)>0$ since $L\not\equiv 0$.
Taking the dual of the diagram, we have
$$\xymatrix{
0&A_1^{\vee}\ar[l] &A^{\vee}\ar[l]_{i^{\vee}} &A_2^{\vee}\ar[l]_{\pi^{\vee}} &0\ar[l]\\
0&A_1^{\vee}\ar[l]\ar[u]_{g^{\vee}} &A^{\vee}\ar[l]_{i^{\vee}}\ar[u]_{f^{\vee}}&A_2^{\vee}\ar[l]_{\pi^{\vee}}\ar[u]_{h^{\vee}} &0\ar[l]
}$$
Suppose $f^{\vee}$ is amplified (resp.~PCD).
Then so are $g^{\vee}$ and $h^{\vee}$.
By induction, so are $g=g^{\vee\vee}$ and $h$.
By Proposition \ref{prop-5lem}, so is $f$.
\end{proof}

Finally, we are able to give another criterion of PCD endomorphisms.
\begin{theorem}\label{thm-pcd-cri2}
Let $f:A\to A$ be a surjective endomorphism of an abelian variety $A$.
Then $f$ is PCD if and only if $f^*D\not\equiv D$ for any nef Cartier divisor $D\not\equiv 0$.
\end{theorem}
\begin{proof}
We may assume $f$ is an isogeny by Lemma \ref{lem-pcd+a}.
Suppose $f$ is not PCD.
Then $f^{\vee}$ is not PCD by Proposition \ref{prop-dual}.
By Proposition \ref{prop-ext} and Lemma \ref{lem-iso-pd}, $\Fix((f^{\vee})^n)$ is infinite.
In particular, there is a positive dimensional abelian subvariety $B^{\vee}\xhookrightarrow[]{p^{\vee}} A^{\vee}$ such that $(f^{\vee})^n|_{B^{\vee}}=\id_{B^{\vee}}$.
Taking the dual, we have an $f^n$-equivariant surjective morphism $p:A\to B$ such that $f^n|_B=\id_B$.
Let $H$ be an ample Cartier divisor on $B$ and $D:=\sum\limits_{i=0}^{n-1}(f^i)^*p^*H$.
Then $D\not\equiv 0$ is a nef Cartier divisor and $f^*D\equiv D$.

Suppose $f^*D\equiv D$ for some nef Cartier divisor $D\not\equiv 0$.
We show by induction on $\dim(A)$ that $f$ is not PCD.
Suppose $D$ is ample.
By \cite[Proposition 2.9]{MZ}, $f$ is of null entropy and hence not PCD (cf.~Lemma \ref{lem-pcd-pos} or Theorem \ref{thm-pcd-cri}).

Suppose $D$ is not ample.
Replacing $D$ by a numerically equivalent class and some multiple, we may assume $D$ is effective and basepoint free (cf.~\cite[Proposition 3.10]{MNW}).
Let $\varphi_{|D|}:A\to X$ be the morphism defined by the linear system $|D|$.
Then $D=\varphi_{|D|}^*H$ for some very ample Cartier divisor $H$ of $X$.
Let $B$ be the neutral component of $K(D)$.
We have $f(B)=B$ and $\dim(B)>0$ by \cite[Application 1, page 60]{Mum}.
Note that $D|_{B+a}\equiv 0$ for any $a\in A$.
Then $\varphi_{|D|}(B+a)$ is a point and hence $\varphi_{|D|}$ factors through the natural quotient map $p_1:A\to A/B$ and $p_2:A/B\to X$ by \cite[Lemma 1.15]{Deb}.
Let $D':=p_2^*H$ and $g:=f|_{A/B}$.
Then $D'\not\equiv 0$ is a nef Cartier divisor of $A/B$ and $g^*D'\equiv D'$.
By induction, $g$ is not PCD and hence $f$ is not PCD by Lemma \ref{lem-pcd-des}.
\end{proof}

\section{Proof of Theorems \ref{main-thm-ab3} and \ref{main-thm-des}}
We first prove Theorem \ref{main-thm-ab3}.

\begin{lemma}\label{lem-pcd-fib}
Let $f:A\to A$ be a PCD endomorphism of an abelian variety $A$.
Then there is no dominant rational map $\pi:A\dasharrow \mathbb{P}^1$ such that $\pi\circ f=f$.
\end{lemma}
\begin{proof}
Suppose such $\pi$ exists.
By the same argument as in the proof of Theorem \ref{thm-amp-orbit}, we have $f^{-1}(X_y)=X_y$ for some general $y\in \mathbb{P}^1$.
Let $B$ be an irreducible component of $X_y$ which is a prime divisor of $A$.
We may assume $f^{-1}(B)=B$ after replacing $f$ by a positive power.
Then $B$ is an abelian variety by Theorem \ref{thm-pcd-PR}.
We may assume $B$ is a subgroup of $A$.
Then there is an $f$-equivariant fibration $p:A\to A/B$.
Note that $A/B$ is an elliptic curve and $f|_{A/B}$ is an (algebraic group) automorphism and hence has finite order.
However, $f|_{A/B}$ is PCD by Proposition \ref{prop-5lem}, a contradiction.
\end{proof}

\begin{proof}[Proof of Theorem \ref{main-thm-ab3}]
(2) and (3) are equivalent by \cite[Theorem 1.2]{GS}.
(1) implies (3) by Lemma \ref{lem-pcd-fib}.

Suppose $f$ is not PCD. By the same argument as in the first part of the proof of Theorem \ref{thm-pcd-cri2}, for some $n>0$, there is an $f^n$-equivariant surjective morphism $p:A\to B$ such that $\dim(B)>0$ and $f^n|_B=\id_B$.
Note that there always exists a dominant rational map $\tau: B\dasharrow \mathbb{P}^1$.
Denote by $\pi:=\tau\circ p$.
Then $\pi\circ f^n=f^n$ and hence $f^n$ has no Zariski dense orbit.
Suppose $O_f(x):=\{f^i(x)\,|\,i\ge 0\}$ is Zariski dense in $A$ for some $x\in A$.
Then $\overline{O_f(x)}=\overline{\bigcup\limits_{j=0}^{n-1}f^j(O_{f^n}(x))}=\bigcup\limits_{j=0}^{n-1}f^j(\overline{O_{f^n}(x)})$ and hence $\overline{O_{f^n}(x)}=A$, a contradiction. 
So (2) implies (1).
\end{proof}

Let $A$ be an abelian variety equipped with a PCD endomorphism $f$. 
In the next two propositions, we characterize $Y$ and $X$ for an $f$-equivariant surjective morphism $A\to Y$ and an $f$-equivariant finite surjective morphism $X\to A$.
They will be used in the proof of Theorem \ref{main-thm-des}.

\begin{proposition}
Let $\pi:A\to Y$ be a surjective morphism of normal projective varieties with $A$ being an abelian variety.
Let $f:A\to A$ and $g:Y\to Y$ be surjective endomorphisms such that $g\circ \pi=\pi\circ f$.
Suppose $f$ is PCD.
Then replacing $f$ by a positive power, there is an $f$-invariant abelian subvariety $B$ of $A$ such that, via Stein factorization, $\pi$ factors through the natural quotient map $p_1:A\to A/B$ and a finite surjective morphism $p_2:A/B\to Y$.
In particular, $g$ is PCD.
\end{proposition}
\begin{proof}
Replacing $f$ by a positive power, we may assume $f$ is an isogeny.
Taking the Stein factorization of $\pi$, we may assume $\pi$ has connected fibres by \cite[Lemma 5.2]{CMZ} and Lemma \ref{lem-fin-lift-des}.
Then the general fibre of $\pi$ is irreducible and we may assume $B:=\pi^{-1}(\pi(0))$ is irreducible.
Note that $f(B)=B$ and Theorem \ref{thm-pcd-PR} implies that $B$ is an abelian variety.
Let $H$ be an ample Cartier divisor of $X$.
By the projection formula, $\pi^*H|_{B+a}\equiv 0$.
Then $\pi(B+a)$ is a point.
By \cite[Lemma 1.15]{Deb}, $p_2:A/B\to Y$ via $p_2(\overline{a})=\pi(a)$ is well defined.
Note that $p_2$ is birational and hence an isomorphism since $A/B$ contains no rational curve (cf.~\cite[Proposition 1.3]{KM}).
Note that $f|_{A/B}$ is PCD by Lemma \ref{lem-pcd-des}.
\end{proof}

\begin{proposition}\label{prop-cov-ab}
Let $\pi:X\to A$ be a finite surjective morphism of normal projective varieties with $A$ being an abelian variety.
Let $f:X\to X$ and $g:A\to A$ be surjective endomorphisms such that $g\circ \pi=\pi\circ f$.
Suppose $g$ is PCD.
Then $\pi$ is \'etale and hence $X$ is an abelian variety.
\end{proposition}
\begin{proof}
Let $n:=\dim(X)$ and $d:=\deg f=\deg g$.
By the ramification divisor formula, $K_X\sim \pi^*K_A+R_{\pi}\sim R_{\pi}$ where $R_{\pi}$ is the ramification divisor of $\pi$.
By the ramification divisor formula again, $K_X\sim f^*K_X+R_f$ where $R_f$ is the ramification divisor of $f$.
So $(f^n)^*R_{\pi}\sim R_{\pi}-\sum\limits_{i=0}^{n-1}(f^i)^*R_f$ for any $n>0$.
Let $H$ be an ample Cartier divisor of $X$.
Suppose $R_f\neq 0$.
Note that $(f^i)^*R_{f}\cdot H^{n-1}$ is a positive integer for each $i\ge 0$.
Then $0<(f^n)^*R_{\pi}\cdot H^{n-1}=R_{\pi}\cdot H^{n-1}-(\sum\limits_{i=0}^{n-1}(f^i)^*R_{f})\cdot H^{n-1}<0$ when $n\gg 1$, a contradiction.
Therefore, $f^*R_{\pi}\sim R_{\pi}$ and hence $f_*R_{\pi}\equiv_w d R_{\pi}$ by the projection formula.
Then $g_*\pi_*R_{\pi}\equiv_w \pi_*f_*R_{\pi}\equiv_w d \pi_*R_{\pi}$.
By the projection formula again, $g^*(\pi_*R_{\pi})\equiv \pi_*R_{\pi}$.
Note that $\pi_*R_{\pi}$ is nef and Cartier.
By Theorem \ref{thm-pcd-cri2}, $R_{\pi}=0$.
Since $A$ is smooth, $\pi$ is then \'etale by the purity of branch loci.
Then $X$ is an abelian variety (cf.~\cite[Section 18, Theorem]{Mum}).
\end{proof}

We recall the following useful decomposition result by Nakayama and Zhang and refer to \cite[Definition 2.9]{Na-Zh} for the definition of weak Calabi-Yau varieties.
Here, we recall a fact that for an abelian variety $A$ and a weak Calabi-Yau variety $S$, the Albanese map $\alb_{A\times S}$ is just the natural projection $p_A:A\times S\to A$.
Note that in the following, we remove the assumption about polarized endomorphisms in the original proposition by applying its proof without the argument related to polarized endomorphisms.

\begin{proposition}\label{prop-nz}(cf.~\cite[Proposition 3.5]{Na-Zh})
Let $f:X\to X$ be a surjective endomorphism of a klt projective variety $X$ with $K_X\sim_{\mathbb{Q}}0$.
Then there exist a finite covering $\tau: A \times S \to X$ \'etale in codimension one for an abelian variety $A$ and a weak Calabi-Yau variety $S$, and surjective endomorphisms $f_A : A \to A$, $f_S : S \to S$ such that $\tau\circ (f_A\times f_S)=f\circ \tau$.
\end{proposition}

\begin{proof}[Proof of Theorem \ref{main-thm-des}]
By \cite[Chapter V, Corollary 4.9]{ZDA}, $K_X\sim_{\Q}0$.
Since $X$ is klt, $X$ has rational singularities and hence the Albanese map is regular (cf.~\cite[Proposition 2.3]{Re} or ~\cite[Lemma 8.1]{Ka}).

By Proposition \ref{prop-nz}, we have the following commutative diagram
$$\xymatrix{
A\times S\ar[d]^{\tau}\ar[r]^{\alb_{A\times S}} & A\ar[d]^{\pi}\\
X\ar[r]^{\alb_X}            & \Alb(X)
}$$
where $\pi$ exists by the universal property of the Albanese map.
The same reason implies that $\pi$ is surjective and hence $\alb_X$ is surjective.
Note that the whole diagram is $(f_A\times f_S)$-equivariant.

Suppose $f$ is quasi-amplified.
Then so are $f_A\times f_S$ and $f_A$ by Lemmas \ref{lem-lift-des} and \ref{lem-prod}.
Since $A$ and $\Alb(X)$ are abelian varieties, $f_A$ is amplified and hence so is $g$ by Proposition \ref{prop-5lem}.
Suppose $f$ is PCD.
Then so are $f_A\times f_S$, $f_A$ and $g$ by Lemmas \ref{lem-fin-lift-des}, \ref{lem-prod} and \ref{lem-pcd-des}.
In both cases, $g$ is PCD (cf.~Theorem \ref{thm-amp-pcd}).
Taking the Stein factorization of $\alb_X$ and applying \cite[Lemma 5.2]{CMZ} and Proposition \ref{prop-cov-ab}, $\alb_X$ has connected fibres by the universal property.
\end{proof}

\par \vskip 1pc
{\bf Acknowledgement.}
The author would like to thank Professor De-Qi Zhang for many inspiring discussions and suggestions to improve this paper.
He would like to thank Professor Daniel Huybrechts for answering his questions about Hyperk\"ahler manifolds and Professor Keiji Oguiso for suggesting Example \ref{exa-k3-1}.
He would also like to thank Professor Paolo Cascini for his very warm hospitality and support
during the preparation of the paper.
Finally, he would like to thank the referee for valuable comments.
The author is supported by a Postdoctoral Fellowship of Max Planck Institute for Mathematics.

\end{document}